\documentclass{amsart}
\usepackage{amsmath}
\usepackage{mathrsfs}
\usepackage{amsthm}
\numberwithin{equation}{section}
\newtheorem{thr}{Theorem}[section]
\theoremstyle{definition}
\newtheorem{defn}[thr]{Definition}
\theoremstyle{remark}
\newtheorem*{rem}{Remark}
\theoremstyle{plain}
\newtheorem{lem}[thr]{Lemma}
\theoremstyle{plain}
\newtheorem{prop}[thr]{Proposition}
\theoremstyle{plain}
\newtheorem{cor}[thr]{Corollary}
\newcommand{\tab}{\mathrm{t}}

\DeclareMathOperator{\CSTD}{\mathrm{CSTD}}
\DeclareMathOperator{\RSTD}{\mathrm{RSTD}}
\DeclareMathOperator{\STD}{\mathrm{STD}}
\DeclareMathOperator{\Tab}{\mathrm{Tab}}

\DeclareMathOperator{\SD}{\mathrm{SD}}
\DeclareMathOperator{\SA}{\mathrm{SA}}
\DeclareMathOperator{\WD}{\mathrm{WD}}
\DeclareMathOperator{\WA}{\mathrm{WA}}
\DeclareMathOperator{\D}{\mathrm{D}}
\DeclareMathOperator{\A}{\mathrm{A}}
\DeclareMathOperator{\Pos}{\mathrm{Pos}}
\newcommand{\row}{\mathrm{row}}
\newcommand{\col}{\mathrm{col}}
\newcommand{\bx}{\mathrm{box}}

\usepackage{young}

\theoremstyle{remark}

\begin{document}

\title[\unboldmath $W\!$-\,graph ideals]{\boldmath\(W\!\)-\,graph ideals}\label{Pre}

\author{\emph{\textbf{Robert B. Howlett and Van Minh Nguyen}}}

\begin{abstract}
We introduce a concept of a \(W\!\)-graph ideal in a Coxeter
group. The main goal of this paper is to describe how to construct
a \(W\!\)-graph from a given \(W\!\)-graph ideal. The
principal application of this idea is in type \(A\), where it
provides an algorithm for the construction of \(W\!\)-graphs for
Specht modules.
\end{abstract}

\maketitle

\section{Introduction}

Let \((W,S)\) be a Coxeter system and \(\mathcal{H}(W)\) its Hecke
algebra over \(\mathbb{Z}[q,q^{-1}]\), the ring of Laurent
polynomials in the indeterminate~\(q\). There are certain
representations of \(\mathcal{H}(W)\) whose structure can be
encoded by combinatorial objects called \(W\!\)-graphs, introduced
by Kazhdan and Lusztig in \cite{kazlus:coxhecke}. A \(W\!\)-graph
provides a compact way of providing all the information needed to
construct the representation. Moreover, from the work of Gyoja,
\cite{gyo:existWG}, it is known that if \(W\) is a finite Weyl group
then all irreducible \(\mathcal{H}(W)\)-modules can be realized as
modules carried by \(W\!\)-graphs. However, the problem of
explicitly describing these \(W\!\)-graphs is not completely
solved.

In  \cite{kazlus:coxhecke} Kazhdan and Lusztig constructed a
special basis for \(\mathcal{H}(W)\), using a family of
polynomials in \(q\) with integer coefficients. These polynomials,
now known as the \textit{Kazhdan-Lusztig polynomials}, are
parametrized by pairs of elements of~\(W\! \), and are defined by
a recursive procedure. The Kazhdan-Lusztig basis gives the regular
representation of \(\mathcal{H}(W)\) a \(W\!\)-graph structure.
Moreover, Kazhdan and Lusztig showed that \(W\!\)-graphs may be
split into cells, which are themselves \(W\!\)-graphs, thus
potentially providing a means of decomposing the regular
representation. In type \(A\) the cells in the regular
\(W\!\)-graph yield irreducible representations; however,
constructing \(W\!\)-graphs for the irreducible representations
has to date been computationally challenging because of the large
number of Kazhdan-Lusztig polynomials that must be calculated.

In \cite{deo:paraKL} Deodhar gave a generalization of the
Kazhdan-Lusztig construction, using \textit{parabolic
Kazhdan-Lusztig polynomials\/} relative to a standard parabolic
subgroup \(W'\!\) to give \(W\!\)-graph structures to
\(\mathcal{H}(W)\)-modules induced from certain one-dimensional
\(\mathcal{H}(W')\)-modules. This raises the question whether
\(W\!\)-graphs for other classes of representations may be
constructed similarly, and to do so is one of the main objectives
of our project. We introduce the concept of a
\textit{\(W\!\)-graph ideal\/} in \((W,\le_L)\) (where \(\le_L\)
is the the partial order such that \(u\le_L v\) if and 
only if \(l(vu^{-1})=l(v)-l(u)\)) and give a
Kazhdan-Lusztig like algorithm to produce, for any such
ideal~\(\mathscr{I}\!\), a \(W\!\)-graph with vertices indexed by
the elements of \(\mathscr{I}\!\).

Our main focus is on \(\mathcal{H}(W_{n})\), the Hecke algebra of
type \(A_{n-1}\). Of course in this case the Weyl group, \(W_n\),
is isomorphic to the symmetric group of degree~\(n\), and its
representation theory (and that of \(\mathcal{H}(W_{n})\)) is
deeply connected with the combinatorics of tableaux. The
irreducibles are parametrized by partitions of \(n\), and for each
partition the corresponding \textit{Specht module\/} has basis in
one-to-one correspondence with the standard tableaux of that
shape. Kazhdan and Lusztig showed in \cite{kazlus:coxhecke} that
for each cell of the Kazhdan-Lusztig \(W\!\)-graph for the left
regular representation of \(\mathcal{H}(W_{n})\), the
Robinson-Schensted algorithm provides a one-to-one correspondence
between the elements of \(W_n\) in the cell and pairs of standard
tableaux with a fixed first term. In \cite{dipjam:heckA} Dipper
and James gave a combinatorial construction of Specht modules.
Attempts have been made to find direct combinatorial constructions
of the \(W\!\)-graphs carried by the cells, but only partial
results have been obtained.

The unpublished draft paper \cite{how:wgraphA} presented a
Kazhdan-Lusztig like algorithm for computing \(W\!\)-graphs for
Specht modules, but the algorithm's correctness was not proved.
The PhD thesis \cite{nguyen:thesis} contains a proof that the
algorithm is indeed correct, and, moreover, can be generalized to
include the construction of \(W\!\)-graphs for modules associated
with skew partitions, as well as Specht modules. The details of
this will be published in another paper. The key fact is that the
set of standard tableaux corresponding to a (skew) partition of
\(n\) is in one-to-one correspondence with an
ideal~\(\mathscr{I}\) in \((W,\le_L)\), and it is shown
that~\(\mathscr{I}\) is a \(W\!\)-graph ideal.

The present paper is organised as follows. The next three sections
present basic definitions and facts concerning Coxeter groups,
Hecke algebras and \(W\!\)-graphs. The notion of a \(W\!\)-graph
ideal is introduced in Section~5, and in Section~6 we present an
illustrative example, constructing a \(W\!\)-graph basis for a
specific Specht module. In Section~7 we prove in general that a
\(W\!\)-graph can be constructed from a \(W\!\)-graph ideal by a
recursive procedure similar to the original Kazhdan-Lusztig
construction, and then in Section~8 we relate our results to the
constructions of Kazhdan-Lusztig and Deohdar.
Finally, in Section~9, we give an alternative construction of a
\(W\!\)-graph induced from the \(W\!\)-graph associated with a
\(W\!\)-graph ideal.

\section{Coxeter groups}\label{Cox}

Let \((W, S)\) be a Coxeter system. Thus \(W\)
is a group generated by a set \(S\) subject to defining relations
of the form
\begin{equation*}
(ss')^{m(s,s')} = 1\quad\text{for all \(s,\,s'\in S\)}
\end{equation*}
where \(m(s,s')=m(s',s)\) is a positive integer or \(\infty\)
and \(m(s,s') = 1\) if and only if \(s = s'\). (A relation
\((ss')^\infty = 1\) is regarded as vacuously true.) Elements of
\(S\) are called \textit{simple reflections}, and the cardinality
of \(S\) is called the \textit{rank\/} of the system. It turns out
that in all cases that \(m(s,s')\) equals the order of \(ss'\) in~\(W\).

Let \(l\) be the length function defined on \(W\); that is, if
\(w\in W\) then \(l(w)\) is the minimal \(k\) such that
\(w=s_{1}s_{2}\cdots s_{k}\) for some elements \(s_1, s_{2},
\ldots, s_{k} \in S\). If \(w=s_{1}s_{2}\cdots s_{k}\) and
\(l(w)=k\), then \(s_{1}s_{2}\cdots s_{k}\) is said to be a
\textit{reduced expression\/} for~\(w\). If \(W\) is finite
then there is a unique longest element in \(W\); we shall denote
it by \(w_S\).

Define \( T = \{\, w^{-1}sw \mid s \in S,\, w \in W\,\}\) (the set
of reflections in \(W\)). The following partial orders are defined
on \(W\).
\begin{defn}[Bruhat order]\label{bruhatorder}
The Bruhat order \(\leq\) is the transitive closure of the
relation \(\buildrel {\scriptscriptstyle T}\over\rightarrow\) 
given by \(u \buildrel{\scriptscriptstyle T}\over \rightarrow w\)
if \(l(u) \leq l(w)\) and \(w = tu\) for some \(t \in
T\cup\{1\}\).
\end{defn}
\begin{defn}[Weak order]\label{leftorder}
The left weak order \(\leq_{L}\) is the transitive closure
of the relation \(\buildrel {\scriptscriptstyle S} \over
\longrightarrow\) given by \(u \buildrel {\scriptscriptstyle S}
\over \longrightarrow w\) if \(l(u) \leq l(w)\) and \(w = su\) for
some \(s \in S\cup\{1\}\). If \(u \leq_{L} w\), we say that \(u\)
is a \textit{suffix\/} of \(w\).
\end{defn}
Observe that \(u \leq_{L} w\) implies \(u \leq w\). It is well
known that if \(W\) is finite then \(u\le_L w_S\) for all \(u\in W\),
where \(w_S\) is the maximal length element of~\(W\).

We shall employ the customary conventions that \(w\ge u\) means
the same thing as \(u\le w\) and that \(u<w\) means \(u\le w\) and
\(u\ne w\), and so forth.

The following property of the Bruhat order is standard (see
\cite[Section 7.4]{hum:coxetreflection}).
\begin{lem}\label{lifting1}
Let \(s \in S\) and \(u,\,w\in W\) satisfy \(u < su\) and \(w <
sw\). Then \(u\le w\) if and only if \(u\le sw\), and \(u\le sw\)
if and only if \(su\le sw\).
\end{lem}

Let \(J\) be an arbitrary subset of \(S\) and \(W_{J}\) the
subgroup of \(W\) generated by~\(J\); such subgroups are called
\textit{standard parabolic subgroups\/} of \(W\). It can be shown
that \((W_{J}, J)\) is a Coxeter system. The length function on
\(W_{J}\) relative to the generating set \(J\) coincides with the
restriction of the length function on \(W\) (see \cite[Section
5.5]{hum:coxetreflection}), and the Bruhat and weak orders
on \(W_J\) are the restrictions of the corresponding orders on
\(W\) (see \cite[Section 5.10]{hum:coxetreflection}). Each
left coset of \(W_{J}\) in \(W\) contains a unique element
of \(D_{J} = \{\,w \in W \mid l(ws) > l(w) \text{ for all }
s \in J\,\}\), and \(l(du) = l(d) +l(u)\) for all \(u \in W_{J}\)
and \(d\in D_J\). The set \(D_{J}\) is called the set of
distinguished (or minimal) left coset representatives in \(W\)
for the subgroup \(W_{J}\) (see \cite[Section 1.10]{hum:coxetreflection}).
If \(W_J\) is finite then we denote the longest element of \(W_{J}\)
by~\(w_{J}\). If \(W\) is finite then we let \(d_{J}\) be
the unique element in \(D_{J} \cap w_{S}W_{J}\); then \(D_{J} =
\{\,w \in W \mid w \leq_{L} d_{J}\,\}\) (see \cite[Lemma
2.2.1]{gecpfei:charHecke}).

\begin{lem}\cite[Lemma 2.1 (iii)]{deo:paraKL}\label{deo1}
Let \(J \subseteq S\). For each \(s \in S\) and each \(w \in
D_{J}\), exactly one of the following occurs:
\begin{itemize}
\item[(i)] \(l(sw) < l(w)\) and \(sw \in D_{J}\); \item[(ii)]
\(l(sw) > l(w)\) and \(sw \in D_{J}\); \item[(iii)] \(l(sw) >
l(w)\) and \(sw \notin D_{J}\), and \(w^{-1}sw \in J\).
\end{itemize}
\end{lem}

We shall find it convenient to make use of the following
definition.

\begin{defn}\label{Pos}
If \(X\subseteq W\), let
\(\Pos(X)=\{\,s\in S\mid l(xs)>l(x)\text{ for all }x\in X\,\}\).
\end{defn}

Thus \(\Pos(X)\) is the largest subset \(J\) of \(S\) such that
\(X\subseteq D_J\).

\section{Hecke algebras}

Let \(\mathcal{A} = \mathbb{Z}[q, q^{-1}]\), the ring of Laurent
polynomials with integer coefficients in the indeterminate \(q\),
and let \(\mathcal{A}^{+} = \mathbb{Z}[q]\). Let \((W, S)\) be a
Coxeter system. Then the corresponding Hecke algebra, denoted \(
\mathcal{H}(W) \), is the associative algebra over \(\mathcal{A}\)
generated by the elements \(\{T_{s} \mid s \in S \}\) subject to
the following defining relations:
\begin{align*}
T^{2}_{s}        &= 1 + (q -q^{-1})T_{s} \quad \text{for all \(s \in S\)},\\
T_{s}T_{s'}T_{s}\cdots &= T_{s'}T_{s}T_{s'}\cdots    \quad
\text{for all \(s, s' \in S\)},
\end{align*}
where in the second of these there are \(m(s,s')\) factors on each
side, \(m(s,s')\) being the order of \(ss'\) in~\(W\). We remark
that the traditional definition has \(T^{2}_{s}=q + (q -1)T_{s}\)
in place of the first relation above; our version is obtained by
replacing \(q\) by \(q^2\) and dividing the generators
by~\(q\).

It is well known that \(\mathcal{H}(W)\) is \(\mathcal{A}\)-free
with an \(\mathcal{A}\)-basis \((\, T_{w}\mid w \in W\,)\) and
multiplication satisfying
\begin{equation*}
\postdisplaypenalty=10000
T_{s}T_{w} = \begin{cases} T_{sw} & \text{if \(l(sw) > l(w)\),}\\
                            T_{sw} + (q - q^{-1})T_{w} & \text{if \(l(sw) < l(w)\).}
              \end{cases}
\end{equation*}
for all \(s\in S\) and \(w\in W\).

If \(J\subseteq S\) then \(\mathcal{H}(W_{J})\), the Hecke algebra
associated with the Coxeter system \((W_{J},J)\), is isomorphic to
the subalgebra of \(\mathcal{H}(W)\) generated by \(\{\,T_{s} \mid
s \in J\,\}\). We shall identify \(\mathcal{H}(W_{J})\) with this
subalgebra.

\section{\(W\!\)-graphs}

Let \(\mathcal{H}=\mathcal{H}(W)\) be the Hecke algebra associated
with the Coxeter system \((W,S)\). Let \(a \mapsto \overline{a}\)
be the involutory automorphism of
\(\mathcal{A}=\mathbb{Z}[q,q^{-1}]\) defined by \(\overline{q}
\mapsto q^{-1}\). This extends to an involution on \(\mathcal{H}\)
satisfying
\begin{equation*}
\overline{T_{s}} = T_{s}^{-1} = T_{s} - (q - q^{-1}) \quad
\text{for all $s \in S$}.
\end{equation*}

A \(W\!\)-graph is a triple \((V, \mu, \tau)\) consisting of a set
\(V\!\), a function \(\mu\colon V \times V \rightarrow
\mathbb{Z}\) and a function \(\tau\) from \(V\) to the power set
of \(S\), subject to the requirement that the free
\(\mathcal{A}\)-module with basis \(V\) admits an
\(\mathcal{H}\)-module structure satisfying
\begin{equation}\label{wgraphdef}
     T_{s}v = \begin{cases}
              -q^{-1}v \quad &\text{if $s \in \tau(v)$}\\
              qv + \sum_{\{u \in V \mid s \in \tau(u)\}}\mu(u,v)u
              \quad &\text{if $s \notin \tau(v)$},
     \end{cases}
\end{equation}
for all \(s \in S\) and \(v \in V\!\).

The set \(V\) is called the vertex set of the \(W\!\)-graph, and
there is a directed edge from a vertex \(v\) to \(u\) if and only
if \(\mu(u,v) \neq 0\). We may regard the integer \(\mu(u,v)\) as
the weight of the edge from \(v\) to \(u\), and the set
\(\tau(v)\) as the colour of the vertex~\(v\).

Since  the \(\mathcal{H}\)-module \(\mathcal{A}V\) is
\(\mathcal{A}\)-free it admits a unique \(\mathcal{A}\)-semilinear
involution \(\alpha \mapsto \overline{\alpha}\) such that
\(\overline v=v\) for all elements \(v\) of the basis~\(V\). It
follows from \eqref{wgraphdef} that for all \(s\in S\) and \(v\in
V\!\),
\begin{align*}
 \overline{T_{s}v} &= \begin{cases}
              -qv \quad &\text{if $s \in \tau(v)$}\\
              q^{-1}v + \sum_{\{u \in V \mid s \in \tau(u)\}}\mu(u,v)u
              \quad &\text{if $s \notin \tau(v)$},\end{cases}\\
&=(T_s-(q-q^{-1}))v\\
&=\overline{T_s}v,
\end{align*}
and hence \(\overline{h\alpha} =\overline{h}\overline{\alpha}\)
for all \(h\in \mathcal{H}\) and \(\alpha \in \mathcal{A}V\!\).

\section{\(W\!\)-graph ideals}\label{section5}

Let \((W,S)\) be a Coxeter system and \(\mathcal{H}\) the
associated Hecke algebra. Let \(\mathscr I\) be an ideal in the
poset \((W,\leq_{L})\); that is, \(\mathscr{I}\) is a subset of
\(W\) such that every \(u\in W\) that is a suffix of an element of
\(\mathscr{I}\) is itself in~\(\mathscr{I}\!\). This condition
implies that \(\Pos(\mathscr{I})=S\setminus\mathscr{I}=\{\,s\in
S\mid s\notin\mathscr{I}\,\}\) (see Definition~\ref{Pos}). Let
\(J\) be a subset of \(\Pos(\mathscr{I})\), so that \(\mathscr
I\subseteq D_J\). For each \(w \in \mathscr I\) we define the
following subsets of \(S\):
\begin{align*}
\SA(w) &= \{\,s \in S \mid sw > w \text{ and } sw \in \mathscr I\,\},\\
\SD(w) &= \{\,s \in S \mid sw < w\,\},\\
\WA_{J}(w) &= \{\,s \in S \mid sw > w \text{ and } sw \in D_{J}
\setminus
\mathscr I\,\},\\
\WD_{J}(w) &= \{\,s \in S \mid sw > w \text{ and } sw \notin
D_{J}\,\}.
\end{align*}
Since  \(\mathscr I\subseteq D_J\) it is clear that, for each \(w
\in \mathscr I\), each \(s \in S\) appears in exactly one of the
four sets defined above. We call the elements of these sets the
strong ascents, strong descents, weak ascents and weak descents
of~\(w\) relative to \(\mathscr{I}\) and \(J\). In contexts where
the set~\(J\) is fixed we frequently omit reference to~\(J\),
writing \(\WA(w)\) and \(\WD(w)\) rather than \(\WA_{J}(w)\) and
\(\WD_{J}(w)\). We also define the sets of descents and ascents of
\(w\) by \(\D_{J}(w) = \SD(w) \cup \WD_{J}(w)\) and  \(\A_{J}(w) =
\SA(w) \cup \WA_{J}(w)\).

\begin{rem}
It follows from Lemma \ref{deo1} that
\begin{align*}
\WA_{J}(w) &= \{\,s \in S \mid sw \notin \mathscr I \text{ and } w^{-1}sw \notin J\,\}\\
\noalign{\vskip-6 pt\hbox{and}\vskip-6 pt} \WD_{J}(w) &= \{\,s \in
S \mid sw \notin \mathscr I \text{ and } w^{-1}sw \in J\,\},
\end{align*}
since \(sw \notin \mathscr I\) implies that \(sw>w\) (given that
\(\mathscr I\) is an ideal in \((W,\le_L)\)). Note also that
\(J=\WD_J(1)\).
\end{rem}

\begin{defn}\label{wgphdetelt}
With the above notation, the set \(\mathscr{I}\) is said to be a
\textit{\(W\!\)-graph ideal\/} with respect to \(J\) if the
following hypotheses are satisfied.
\begin{itemize}
\item[(i)] There exists an \(\mathcal{A}\)-free
\(\mathcal{H}\)-module \(\mathscr{S}=\mathscr{S}(\mathscr{I},J)\)
possessing an \(\mathcal{A}\)-basis \(B = (\,b_{w} \mid w \in
\mathscr I\,)\) on which the generators \(T_{s}\) act by
\begin{equation*}
T_{s}b_{w} =
\begin{cases}
  b_{sw}  & \text{if \(s \in \SA(w)\),}\\
  b_{sw} + (q - q^{-1})b_{w} & \text{if \(s \in \SD(w)\),}\\
  -q^{-1}b_{w} & \text{if \(s \in \WD_{J}(w)\),}\\
  qb_{w} - \sum\limits_{\substack{y \in \mathscr I\\y < sw}} r^s_{y,w}b_{y} &
  \text{if \(s \in \WA_{J}(w)\),}
\end{cases}
\end{equation*}
for some polynomials \(r^s_{y,w} \in q\mathcal{A}^{+}\!\).
\item[(ii)] The module \(\mathscr{S}\) admits an
\(\mathcal{A}\)-semilinear involution \(\alpha \mapsto
\overline{\alpha}\) satisfying \(\overline{b_1}=b_1\) and
\(\overline{h\alpha} =\overline{h}\overline{\alpha}\) for all
\(h\in \mathcal{H}\) and \(\alpha \in \mathscr{S}\).
\end{itemize}
\end{defn}

We shall show in Section~\ref{section7} below that if
\(\mathscr{I}\) is a \(W\!\)-graph ideal with respect to \(J\)
then the associated module \(\mathscr{S}(\mathscr{I},J)\) is
isomorphic to a \(W\!\)-graph module. Moreover, the \(W\!\)-graph
can be constructed by an algorithm that depends only on
\(\mathscr{I}\) and~\(J\). Hence \(\mathscr{S}(\mathscr{I},J)\) is
determined up to isomorphism by~\(\mathscr{I}\) and~\(J\).

\begin{rem}
As we shall see in Section~\ref{klparrel} below, it is quite
possible for an ideal \(\mathscr{I}\) to be a \(W\!\)-graph ideal
with respect to two different subsets~\(J\) of
\(\Pos(\mathscr{I})\), corresponding to two \(W\!\)-graph modules
that are not isomorphic. So the set \(J\) is an important part of
the definition of a \(W\!\)-graph ideal.
\end{rem}

\begin{defn}
If \(\Lambda\subseteq W\) and \(\mathscr{I}=\{\,u\in W\mid
u\le_Lw\text{ for some \(w\in\Lambda\)}\,\}\) is a \(W\!\)-graph
ideal then we call \(\Lambda\) a \(W\!\)\textit{-graph determining
set}, and we call \(w\in W\) a \(W\!\)\textit{-graph determining
element\/} if
 \(\{w\}\) is a \(W\!\)-graph determining set.
 \end{defn}

The simplest example of a \(W\!\)-graph determining element is
\(w_S\), the maximal length element of a finite Coxeter
group~\(W\!\), with \(J\) the empty subset of~\(S\). The \(W\!\)-graph
we obtain is the
Kazhdan-Lusztig \(W\!\)-graph corresponding to the regular
representation of~\(W\). More generally, if \(J\) is an arbitrary
subset of \(S\) then \(d_{J}\), the minimal length element of the
left coset \(w_SW_{J}\), is a \(W\!\)-graph determining element
with respect to~\(J\) and also with respect to~\(\emptyset\). In
both cases \(\mathscr I=D_{J}\), and we recover Deodhar's
parabolic analogues of the Kazhdan-Lusztig construction. See
Section~\ref{klparrel} below for the details.

\section{An example}\label{section6}

The general algorithm for constructing \(W\!\)-graphs from
\(W\!\)-graph ideals is deferred to the next section. In the
current section we present a motivational example.

Let \(W_{n}\) be the Coxeter group of type \(A_{n-1}\), which we
identify with the the symmetric group on \([1,n]\), the set of
integers from 1 to~\(n\), by identifying the simple reflections
\(s_1,\,s_2,\,\ldots,\,s_{n-1}\) with the transpositions
\((1,2),\,(3,4),\,\ldots,\,(n-1,n)\) (respectively). We use a
left-operator convention for permutations, writing \(wi\) for the
action of \(w\in W_n\) on \(i\in[1,n]\). It is well known that if
\(t=(i,j)\in W_n\) is an arbitrary transposition, with \(i<j\),
and \(w\in W_n\) is an arbitrary permutation, then \(wt<w\) if and
only if \(wi>wj\) and \(tw<w\) if and only if \(w^{-1}i>w^{-1}j\);
moreover, \(l(w)\) is the number of pairs \((i,j)\in
[1,n]\times[1,n]\) such that \(i<j\) and \(wi>wj\).

Since our example will involve Young diagrams and tableaux, we
need to start by recalling some basic definitions and establishing
our notation.

A sequence of positive integers \(\lambda =
(\lambda_{1},\lambda_{2} \ldots, \lambda_{k})\) is called a
partition of \(n\) if \(\lambda_{1} + \lambda_{2} + \cdots +
\lambda_{k} = n\) and \(\lambda_{1} \ge \cdots \ge \lambda_{k}\).
The \(\lambda_i\) are called the \textit{parts\/} of \(\lambda\).
We define \(P(n)\) to be the set of all
partitions of~\(n\). For each \(\lambda = (\lambda_{1}, \ldots,
\lambda_{k}) \in P(n)\) we define
\[
[\lambda] = \{\,(i,j) \mid 1 \leq j \leq \lambda_{i}\text{ and } 1
\leq i \leq k\,\},
\]
and refer to this as the Young diagram of~\(\lambda\). Pictorially
\([\lambda]\) is represented by a left-justified array of boxes
with \(\lambda_{i}\) boxes in the \(i\)-{th} row; the pair
\((i,j)\in[\lambda]\) corresponds to the \(j\)-{th} box in the
\(i\)-{th} row. Thus the Young diagram of \(\lambda =
(4,2,2)\) looks like this:
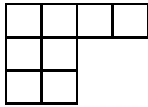
\[  \vcenter{\begin{Young}
        &&&\cr
        &\cr
        &\cr
      \end{Young}}
\]

If \(\lambda\) is a partition of \(n\) then a
\(\lambda\)\textit{-tableau} is a bijection \(t\colon[\lambda]
\rightarrow [1,n] \). In other words, \(t\) is a one to one
correspondence between the boxes of the Young diagram
\([\lambda]\) and the integers from 1 to~n. Of course \(t\) can be
conveniently described by writing the number \(t(i,j)\) in the box
\((i,j)\), for all \((i,j)\in[\lambda]\). For each \(i\in[1,n]\)
we define \(\row_t(i)\) and \(\col_t(i)\) to be the row index and
column index of \(i\)~in~\(t\) (so that
\(t^{-1}(i)=(\row_t(i),\col_t(i)\))). We define \(\Tab(\lambda)\)
to be the set of all \(\lambda\)-tableaux, and we let
\(\tab^\lambda\) be the specific \(\lambda\)-tableau given by
\[
\tab_\lambda(i,j)=j+\sum_{h=1}^{i-1}\lambda_h
\]
for all \((i,j)\in[\lambda]\). That is, the numbers
\(1,\,2,\,\dots,\,\lambda_1\) fill the first row of
\([\lambda]\) in order from left to right, then the numbers
\(\lambda_1+1,\,\lambda_1+2,\,\dots,\,\lambda_1+\lambda_2\)
similarly fill the second row, and so on. 

We define \(\tab_{\lambda}\) to be
the \(\lambda\)-tableau that is the transpose of the
\(\lambda'\)-tableau \(\tab^{\lambda'}\), where \(\lambda'\) is
the partition dual to~\(\lambda\). Thus \(\tab_{\lambda}\) is the
unique standard \hbox{\(\lambda\)-tableau} whose columns consist of
sequences of consecutive numbers, while \(\tab^{\lambda}\) is the
unique standard \(\lambda\)-tableau whose rows consist of
sequences of consecutive numbers. We shall find it
convenient to define \(\bx_\lambda(i)=\tab_\lambda^{-1}(i)\); thus
\(\bx_\lambda(i)\) is the box of \([\lambda]\) such that \(i\) is
in \(\bx_\lambda(i)\) in~\(\tab_\lambda\). We say that
\(\bx_\lambda(i)\) is ``earlier'' than \(\bx_\lambda(j)\) if
\(i<j\).

It is clear that for any fixed \(\lambda\in P(n)\) the group
\(W_{n}\) acts on the set of all \(\lambda\)-tableaux, via
\((wt)(i,j) = w(t(i,j))\) for all \((i,j)\in[\lambda]\), for all
\(\lambda\)-tableaux \(t\) and all \(w \in W_{n}\). Moreover, the
map from \(W_n\) to \(\Tab(\lambda)\) defined by \(w\mapsto
w\tab_\lambda\) for all \(w\in W_n\) is bijective. We use this
bijection to transfer the partial orders defined in Definitions
\ref{bruhatorder} and \ref{leftorder} from \(W_n\) to
\(\Tab(\lambda)\). Thus if \(t_1,\,t_2\) are arbitrary
\(\lambda\)-tableaux  and we write \(t_1=w_1\tab_\lambda\) and
\(t_2=w_2\tab_\lambda\) with \(w_1,\,w_2\in W_n\), then by
definition \(t_1\le t_2\) if and only if \(w_1\le w_2\), and
\(t_1\le_L t_2\) if and only if \(w_1\le_L w_2\). Similarly, if
\(t=w\tab_\lambda\) is an arbitrary \(\lambda\)-tableau, where
\(w\in W_n\), then we define \(l(t)=l(w)\).

For later reference, we note the following trivial result.

\begin{lem}\label{tabincrease} Let \(w\in W_n\) and let
\(t=w\tab_\lambda\) be the corresponding \(\lambda\)-tableau. If
\(i\in[1,n-1]\) then \(l(s_it)>l(t)\) if and only if either
\(\col_t(i)<\col_t(i+1)\) or \(\col_t(i)=\col_t(i+1)\) and
\(\row_t(i)<\row_t(i+1)\).
\end{lem}

\begin{proof}
Observe that \(w^{-1}i=w^{-1}(t(\row_t(i),\col_t(i)))
=\tab_\lambda(\row_t(i),\col_t(i))\), and similarly
\(w^{-1}(i+1)=\tab_\lambda(\row_t(i+1),\col_t(i+1))\). Since
\(\tab_\lambda(j,k)<\tab_\lambda(j',k')\) if and only if either
\(k<k'\) or \(k=k'\) and \(j<j'\), the condition that
\(\col_t(i)<\col_t(i+1)\) or \(\col_t(i)=\col_t(i+1)\) and
\(\row_t(i)<\row_t(i+1)\) is equivalent to
\(w^{-1}i<w^{-1}(i+1)\). Since this in turn is equivalent to
\(l(s_iw)>l(w)\), the result follows.
\end{proof}

A \(\lambda\)-tableau \(t\), where \(\lambda \in P(n)\), is said
to be \textit{column standard\/} if its entries increase down the
columns, that is, if \(t(i,j) < t(i+1,j)\) whenever
\((i,j)\in[\lambda]\) and \((i+1,j)\in[\lambda]\). Similarly, \(t\)
is said to be \textit{row standard\/} if
its entries increase along the rows, that is, if \(t(i,j) <
t(i,j+1)\) whenever \((i,j)\in[\lambda]\) and
\((i,j+1)\in[\lambda]\). A \textit{standard tableau\/} is a
tableau that is both column standard and row standard. We write
\(\CSTD(\lambda)\), \(\RSTD(\lambda)\) and \(\STD(\lambda)\) for
the sets of all column standard tableaux, row standard tableaux
and standard tableaux for \(\lambda\).

Given \(\lambda\in P(n)\) we define \(J_\lambda\) to be the subset
of \(S\) consisting of those simple reflections \(s_i=(i,i+1)\)
such that \(i\) and \(i+1\) lie in the same column
of~\(\tab_\lambda\), and we define \(W_\lambda\) to be the
standard parabolic subgroup of \(W_n\) generated by~\(J_\lambda\).
Thus, by our convention, \(W_\lambda\) is the column stabilizer
of~\(\tab_\lambda\) rather than the row stabilizer of \(\tab^\lambda\).
Moreover, the set of minimal left coset representatives for
\(W_\lambda\) in \(W_n\) is the set
\[
D_\lambda=\{\,d\in W_n\mid di<d(i+1)\text{ whenever \(s_i\in
J_\lambda\)}\,\}
\]
since the condition \(di<d(i+1)\) is equivalent to
\(l(ds_i)>l(d)\). It follows that \(\{\,d\tab_\lambda\mid d\in
D_\lambda\,\}\) is precisely the set of column standard
\(\lambda\)-tableaux.

Now suppose that \(t\in\STD(\lambda)\) and \(t \neq \tab^{\lambda}\).
Choose \(i\) to be the
least integer whose position in \(t\) is not the same as its
position in~\(\tab^\lambda\), and let
\(j=t(\row_{\tab^\lambda}(i),\col_{\tab^\lambda}(i))\), the number
whose position in \(t\) is the position of~\(i\)
in~\(\tab^\lambda\). If \(h=\row_t(j)\) then the number \(j-1\)
cannot appear to the left of~\(j\) in the \(h\)-th row of~\(t\),
or in any earlier row, since these positions are occupied by the
numbers from 1~to~\(i-1\). Hence, since \(t\) is standard, it
follows that \(\row_t(j-1)>\row_t(j)\) and
\(\col_t(j-1)<\col_t(j)\). In particular, since \(j-1\) and \(j\)
are not in the same row of \(t\) or the same column of~\(t\), the
tableau obtained from \(t\) by swapping the positions of \(j-1\)
and \(j\) is still standard. That is,
\(s_{j-1}t\in\STD(\lambda)\). But by Lemma~\ref{tabincrease} above
we see that \(l(s_{j-1}t)>l(t)\), and therefore \(t<_L s_{j-1}t\).
So \(t\) is not maximal in the ordering~\(<_L\), and it follows
that \(\tab^\lambda\) is the unique maximal standard
\(\lambda\)-tableau relative to~\(<_L\).

Similarly, if \(t\in\STD(\lambda)\) and \(s_jt<t\) for some
\(j\in[1,n-1]\), then \(t\) has \(j+1\) in an earlier box
than~\(j\), and since \(t\) is standard we see that
\(\row_t(j+1)>\row_t(j)\) and \(\col_t(j+1)<\col_t(j)\). Thus
\(s_jt\in\STD(\lambda)\). So if \(t'\) is any \(\lambda\)-tableau
such that \(t'<_Lt\) then \(t'\) is standard. Hence we obtain the
following result (see \cite[Lemma 1.5]{dipjam:heckA}).
\begin{lem}\label{characterisetstd1}
Let \(\lambda \in P(n)\) and define \(v_\lambda\in W_n\) by the
requirement that \(\tab^\lambda=v_\lambda\tab_\lambda\). Then
\(\STD(\lambda)=\{\,w\tab_\lambda \mid w\le_L v_\lambda\,\} =
\{\,t \in \Tab(\lambda) \mid t\le_L \tab^{\lambda}\,\}\).
\end{lem}

For \(t \in \STD(\lambda)\), define
\begin{align*}
\SA(t) &= \{\,i \in [1,n-1] \mid \col_{t}(i) < \col_{t}(i+1)\text{
and }
\row_{t}(i) \neq \row_{t}(i+1)\,\},\\
\SD(t) &= \{\,i \in [1,n-1] \mid \col_{t}(i) > \col_{t}(i+1)\,\},\\
\WA(t) &= \{\,i \in [1,n-1] \mid \row_{t}(i) = \row_{t}(i+1)\,\},\\
\WD(t) &= \{\,i \in [1,n-1] \mid \col_{t}(i) = \col_{t}(i+1)\,\}.
\end{align*}
Observe that if \(\mathscr I\) is the left ideal of
\((W_n,\le_L)\) generated by \(v_\lambda\) and if \(J=J_\lambda\),
then for each \(w\in\mathscr I\) the sets
\(\SA(w),\,\SD(w),\,\WA_J(w)\) and \(\WD_J(w)\) as defined in
Section~\ref{section5} coincide with the sets
\(\SA(w\tab_\lambda),\,\SD(w\tab_\lambda),\,\WA(w\tab_\lambda)\)
and \(\WD(w\tab_\lambda)\) as defined above.

Let \(\mathcal{H}_n=\mathcal{H}(W_{n})\) be the Hecke algebra of
\(W_n\). Thus \(\mathcal{H}_n\) is generated by elements
\(T_1,\,T_2,\,\ldots,\,T_{n-1}\) satisfying
\(T_iT_{i+1}T_i=T_{i+1}T_iT_{i+1}\) for all \(i\in[1,n-2]\) and
\(T_iT_j=T_jT_i\) for all \(i,\,j\in[1,n-1]\) with \(|i-j|>1\), as
well as \(T_i^2=1+(q-q^{-1})T_i\) for all~\(i\in[1,n-1]\). Let
\(\lambda \in P(n)\) and let \(S^{\lambda}\) be the Specht module
for \(\mathcal{H}_n\) corresponding to \(\lambda\). It follows
from results proved in \cite[Chapter 3]{math:heckeA} that
\(S^{\lambda}\) has an \(\mathcal{A}\)-basis \((\,b_{t} \mid t
\in\STD(\lambda)\,)\) such that for all \(i \in [1,n-1]\) and \(t
\in \STD(\lambda)\),
\begin{equation}
T_{i}b_{t} =
\begin{cases}\label{stdbasis}
  b_{s_{i}t}  & \text{if \(i \in \SA(t)\),}\\
  b_{s_{i}t} + (q - q^{-1})b_{t} & \text{if \(i \in \SD(t)\),}\\
  -q^{-1}b_{t} & \text{if \(i \in \WD(t)\),}\\
  qb_{t} - \sum\limits_{s <t} r^{(i)}_{s,t}b_{s} &
\text{if \(i \in \WA(t)\),}
\end{cases}
\end{equation}
where the \(r^{(i)}_{s,t}\) in the last equation are in
\(\mathcal{A}\), but are not easy to describe explicitly.

The basis \((\,b_{t} \mid t \in \STD(\lambda)\,)\) is known as the
\textit{standard basis\/} of \(S^{\lambda}\). Note that our
hypotheses and conventions are slightly different from those used
in \cite{math:heckeA}, and hence our formulas above are also
slightly different from those in~\cite{math:heckeA}. More
explanation of \eqref{stdbasis} will be given below.

Let \(\mathcal{F}\) be the field of fractions of \(\mathcal{A}\),
and write \(\mathcal{F}S^\lambda\) for the \(\mathcal{F}\)-module
obtained from \(S^\lambda\) by extension of scalars. In this
context we can obtain the simpler \textit{seminormal form\/} of
the representation: \(\mathcal{F}S^\lambda\) has an
\(\mathcal{F}\)-basis \((\,b'_{t} \mid t \in \STD(\lambda)\,)\)
such that for all \(i\in[1,n-1]\) and \(t \in \STD(\lambda)\),
\begin{equation*}
T_{i}b'_{t} =
\begin{cases}
   -q^{-1}b'_{t} & \text{if \(i \in\WD(t)\),}\\
   qb'_{t} & \text{if \(i \in \WA(t)\),}\\
   p_{1}(d;q)b'_{t} +
   p_{2}(d;q)b'_{s_{i}t}& \text{otherwise},
\end{cases}
\end{equation*}
where \(d = (x_{1} - y_{1}) - (x_{2} - y_{2})\) if the row and
column indices of \(i\) and \(i+1\) in \(t\) are, respectively,
\(x_{1}\) and \(y_{1}\) and \(x_{2}\) and \(y_{2}\), and
\begin{align*}
p_{1}(d;q) &= (q^{2} - 1)/ (q - q^{2d+1}),\\
p_{2}(d;q) &= (1 - q^{2d+2})/(q - q^{2d+1}).
\end{align*}
A proof of the validity of these formulas can be found in the
paper of Ariki and Koike, \cite[Theorem 3.7]{arike:koikeHofWr}.
Note that we are using a variant of \(\mathcal{H}_n\) in which the
eigenvalues of the generators \(T_{i}\) are \(q\) and \(q^{-1}\),
whereas Ariki and Koike use the traditional \(q\) and \(-1\);
hence to convert the formulas Ariki and Koike give to the ones
that are appropriate for our context it was necessary to replace
\(q\) by \(q^{2}\) and \(T_{i}\) by \(qT_{i}\).

The seminormal form suffers the drawback that it gives matrix
coefficients that are not integral. The standard basis and the
\(W\!\)-graph basis \((\,c_{t} \mid t \in \STD(\lambda)\,)\) give
integral representations but no (currently known) simple formulae
for all the matrix coefficients. All three bases are related by
triangular basis changes, with \(c_{t_{\lambda}} = b_{t_{\lambda}}
= b'_{t_{\lambda}}\). (This vector spans the \(1\)-dimensional
subspace of \(S^{\lambda}\) consisting of those \(v\) such that
\(T_{w}v = (-q)^{-l(w)}v\) for all \(w\in W_\lambda\).)

Using the seminormal form of the representation it can easily be
shown that \(\mathcal{F}S^\lambda\) admits a semilinear involution
\(v\mapsto \overline v\) satisfying
\(\overline{hv}=\overline{h}\overline{v}\) for all
\(h\in\mathcal{H}_n\) and all \(v\in S^\lambda\). Indeed, if
\(v\in S^{\lambda}\) then \(v = \sum_{t \in \STD(\lambda)}
a_{t}b'_{t}\) for some coefficients \(a_t\in\mathcal{F}\), and we
define \(\overline{v} = \sum_{t \in
\STD(\lambda)}\overline{a_{t}}b'_{t}\). Then for all \(i \in
[1,n-1]\) and \(t \in \STD(\lambda)\),
\begin{equation*}
\overline{T_{i}}b'_{t} = T_{i}b'_{t} + (q^{-1} - q)b'_{t} =
\overline{T_{i}b'_{t}}
\end{equation*}
since \(p_{1}(d;q) + (q^{-1} - q) = p_{1}(d;q^{-1})\) and
\(p_{2}(d;q) = p_{2}(d;q^{-1})\). It follows by linearity that
\(\overline{T_iv} = \overline{T_i}\overline{v}\) for all \(i \in
[1,n-1]\) and all \(v\in\mathcal{F}S^\lambda\), and this gives the
desired result since the \(T_i\) generate~\(\mathcal{H}_n\).

Now for our example. We take \(n=7\) and let \(\lambda=(3,3,1)\),
a partition of 7 giving a Specht module of dimension~21. The 21
standard tableaux \(t_1,\,t_2,\,\ldots,t_{21}\) are listed in
order below.
\begin{gather*}
\begin{Young}
1&4&6\cr 2&5&7\cr 3\cr
\end{Young}\hskip 6pt plus 1 pt minus 1 pt
\begin{Young}
1&3&6\cr 2&5&7\cr 4\cr
\end{Young}\hskip 6pt plus 1 pt minus 1 pt
\begin{Young}
1&2&6\cr 3&5&7\cr 4\cr
\end{Young}\hskip 6pt plus 1 pt minus 1 pt
\begin{Young}
1&3&6\cr 2&4&7\cr 5\cr
\end{Young}\hskip 6pt plus 1 pt minus 1 pt
\begin{Young}
1&2&6\cr 3&4&7\cr 5\cr
\end{Young}\hskip 6pt plus 1 pt minus 1 pt
\begin{Young}
1&4&5\cr 2&6&7\cr 3\cr
\end{Young}\hskip 6pt plus 1 pt minus 1 pt
\begin{Young}
1&3&5\cr 2&6&7\cr 4\cr
\end{Young}\\
\begin{Young}
1&2&5\cr 3&6&7\cr 4\cr
\end{Young}\hskip 6pt plus 1 pt minus 1 pt
\begin{Young}
1&3&4\cr 2&6&7\cr 5\cr
\end{Young}\hskip 6pt plus 1 pt minus 1 pt
\begin{Young}
1&2&4\cr 3&6&7\cr 5\cr
\end{Young}\hskip 6pt plus 1 pt minus 1 pt
\begin{Young}
1&2&3\cr 4&6&7\cr 5\cr
\end{Young}\hskip 6pt plus 1 pt minus 1 pt
\begin{Young}
1&3&5\cr 2&4&7\cr 6\cr
\end{Young}\hskip 6pt plus 1 pt minus 1 pt
\begin{Young}
1&2&5\cr 3&4&7\cr 6\cr
\end{Young}\hskip 6pt plus 1 pt minus 1 pt
\begin{Young}
1&3&4\cr 2&5&7\cr 6\cr
\end{Young}\\
\begin{Young}
1&2&4\cr 3&5&7\cr 6\cr
\end{Young}\hskip 6pt plus 1 pt minus 1 pt
\begin{Young}
1&2&3\cr 4&5&7\cr 6\cr
\end{Young}\hskip 6pt plus 1 pt minus 1 pt
\begin{Young}
1&3&5\cr 2&4&6\cr 7\cr
\end{Young}\hskip 6pt plus 1 pt minus 1 pt
\begin{Young}
1&2&5\cr 3&4&6\cr 7\cr
\end{Young}\hskip 6pt plus 1 pt minus 1 pt
\begin{Young}
1&3&4\cr 2&5&6\cr 7\cr
\end{Young}\hskip 6pt plus 1 pt minus 1 pt
\begin{Young}
1&2&4\cr 3&5&6\cr 7\cr
\end{Young}\hskip 6pt plus 1 pt minus 1 pt
\begin{Young}
1&2&3\cr 4&5&6\cr 7\cr
\end{Young}
\end{gather*}
Note that we have chosen a total ordering of \(\STD(\lambda)\)
that is consistent with the partial ordering~\(\le\), in the sense
that if \(i\le j\) then \(t_i\le t_j\). Let
\(b_1,\,b_2,\,\ldots,\,b_{21}\) be the standard basis elements
corresponding (respectively) to \(t_1,\,t_2,\,\ldots,\,t_{21}\).
We shall construct a new basis \(c_1,\,c_2,\,\ldots,\,c_{21}\)
such that for all~\(j\),
\[
c_j=b_j-q\sum_{i<j}f_{i,j}c_i
\]
for certain \(f_{i,j} \in \mathbb{Z}[q]\), to be defined
recursively. In terms of this new basis the action of the algebra
will be as follows:
\begin{equation}\label{wgformulas}
T_kc_j=
\begin{cases}
-q^{-1}c_j&\text{if \(k\in \D(t_j)\),}\\
qc_j+\sum_{i\in \mathcal R(k,j)} \mu_{i,j}c_i&\text{if \(k\in \WA(t_j)\),}\\
qc_j+c_h+\sum_{i\in \mathcal R(k,j)} \mu_{i,j}c_i&\text{if \(k\in
\SA(t_j)\),}
\end{cases}
\end{equation}
where \(h\) is defined by \(s_kt_j=t_h\), the set \(\mathcal
R(k,j)\) consists of all \(i<j\) such that \(k\) is a descent
of~\(t_i\), and \(\mu_{i,j}\) is the constant term of~\(f_{i,j}\).

These conditions easily yield formulas for the \(c_j\), as listed
below. To start the process, \(c_1=b_1\) is given. Now to find
\(c_2\), we first find a strong descent \(r\) of \(t_2\); in this
case, the only choice is \(r=3\). By the third formula above we
must have \(T_3c_1=qc_1+c_2\), and thus
\(c_2=T_3b_1-qc_1=b_2-qc_1\). In general, to find \(c_h\) given
that the earlier \(c_j\)'s have already been found, first find
\(k\in \SD(t_h)\), and let \(t_j=s_kt_h\). Then
\begin{align*}
c_h&=T_kc_j-qc_j-\sum_{i\in \mathcal R(k,j)}\mu_{i,j}c_i\\
&=T_k\bigl(b_j-q\sum_{i<j}f_{i,j}c_i\bigr)-qc_j-\sum_{i\in \mathcal R(k,j)}\mu_{i,j}c_i\\
&=b_h-qc_j-q\sum_{i<j}f_{i,j}T_kc_i-\sum_{i\in \mathcal
R(k,j)}\mu_{i,j}c_i,
\end{align*}
which can be expressed in the form
\(b_h-q\sum_{\ell<h}f_{\ell,h}c_\ell\) by using the formulas for
evaluating \(T_kc_i\). The crucial point is that the coefficient
of each \(c_\ell\) in \(-qf_{i,j}T_kc_i\) will be a polynomial
divisible by \(q\) unless \(i\in \mathcal R(k,j)\), in which case
\(-qf_{i,j}T_kc_i=f_{i,j}c_i\), and the constant term
\(\mu_{i,j}c_i\) is cancelled by one of the terms in the second
sum. In this way all the terms in the second sum also disappear.

For example, having found \(c_2\), to find \(c_3\) we first
observe that 2 is a descent of \(t_3\) and \(s_2t_3=t_2\), giving
\[
c_3=b_3-qc_2-q\sum_{i<2}f_{i,2}T_2c_i-\sum_{i\in \mathcal
R(2,2)}\mu_{i,2}c_i.
\]
Since \(f_{1,2}=\mu_{1,2}=1\) and \(2\in \D(t_1)\) we find that
\(-qf_{1,2}T_2c_1=c_1=\mu_{1,2}c_1\), leaving \(c_3=b_3-qc_2\).
After similarly calculating that \(c_4=b_4-qc_2\), the calculation
for \(c_5\) proceeds as follows. Since \(r=2\) is a descent of
\(t_5\) with \(s_2t_5=t_4\),
\begin{align*}
c_5&=b_5-qc_4-qT_2c_2-\sum_{i\in \mathcal R(2,4)}\mu_{i,4}c_i,\\
&=b_5-qc_4-q(qc_2+c_3+c_1)-0,
\end{align*}
since the fact that \(2\notin \D(t_2)\) means that \(\mathcal
R(2,4)\) is empty, and \(T_2c_2=qc_2+c_3+c_1\) (since \(2\in
D(t_1)\) and \(\mu_{1,2}=1\), and \(s_2t_2=t_3\)). As a further
example, the calculations involved in deriving the formula for
\(c_{21}\) are given below.
\begin{align*}
c_1&=b_1,\\
c_2&=b_2-qc_1,\\
c_3&=b_3-qc_2,\\
c_4&=b_4-qc_2,\\
c_5&=b_5-qc_4-qc_3-q^2c_2-qc_1,\\
c_6&=b_6-qc_1,\displaybreak[0]\\
c_7&=b_7-qc_6-qc_2-q^2c_1,\displaybreak[0]\\
c_8&=b_8-qc_7-qc_3-q^2c_2,\displaybreak[0]\\
c_9&=b_9-qc_7-q^2c_6-qc_4-q^2c_2-qc_1,\displaybreak[0]\\
c_{10}&=b_{10}-qc_9-qc_8-q^2c_7-qc_5-q^2c_4-q^2c_3-q^3c_2-q^2c_1,\displaybreak[0]\\
c_{11}&=b_{11}-qc_{10}-q^2c_9-q^2c_5-qc_4-q^3c_1,\displaybreak[0]\\
c_{12}&=b_{12}-qc_7-qc_4-q^2c_2,\displaybreak[0]\\
c_{13}&=b_{13}-qc_{12}-qc_8-q^2c_7-qc_6-qc_5-q^2c_4-q^2c_3-q^3c_2-q^2c_1,\displaybreak[0]\\
c_{14}&=b_{14}-qc_{12}-qc_9-q^2c_7-q^2c_4-q^3c_2-q^2c_1,\displaybreak[0]\\
c_{15}&=b_{15}-qc_{14}-qc_{13}-q^2c_{12}-qc_{10}-q^2c_9-q^2c_8-q^3c_7-q^2c_6\\[-2pt]
&\qquad-q^2c_5-q^3c_4-q^3c_3-q^4c_2-q^3c_1,\\
c_{16}&=b_{16}-qc_{15}-qc_{14}-q^2c_{13}-qc_{12}-qc_{11}-q^2c_{10}-q^3c_9-qc_8\\[-2pt]
&\qquad-q^2c_7-q^3c_6-q^3c_5-q^2c_4-q^4c_1,\\
c_{17}&=b_{17}-qc_{12}-q^2c_7,\\
c_{18}&=b_{18}-qc_{17}-qc_{13}-q^2c_{12}-q^2c_8-q^3c_7-q^2c_6,\\
c_{19}&=b_{19}-qc_{17}-qc_{14}-q^2c_{12}-q^2c_9-q^3c_7-qc_4,\\
c_{20}&=b_{20}-qc_{19}-qc_{18}-q^2c_{17}-qc_{15}-q^2c_{14}-q^2c_{13}-q^3c_{12}-q^2c_{10}\\[-2pt]
&\qquad-q^3c_9-q^3c_8-q^4c_7-q^3c_6-qc_5-q^2c_4-q^2c_1,\\
c_{21}&=b_{21}-qc_{20}-q^2c_{19}-q^2c_{18}-qc_{17}-qc_{16}-q^2c_{15}-q^3c_{14}-q^3c_{13}\\[-2pt]
&\qquad-q^2c_{12}-q^2c_{11}-q^3c_{10}-q^4c_9-q^2c_8-q^3c_7-q^4c_6-q^2c_5\\[-2pt]
&\qquad-(q^3+q)c_4-qc_3-q^2c_2-q^3c_1.
\end{align*}

Here are the calculations for \(c_{21}\). We have
\(s_3t_{21}=t_{20}\); so
\begin{align*}
c_{21}&=b_{21}-qc_{20}-\sum_{i<20}f_{i,20}T_3c_i-\!\!\sum_{i\in \mathcal R(3,20)}\!\!\mu_{i,20}c_i\\
&=b_{21}-qc_{20}-qT_3c_{19}-qT_3c_{18}-q^2T_3c_{17}-qT_3c_{15}-q^2T_3c_{14}\\[-2 pt]
&\qquad-q^2T_3c_{13}-q^3T_3c_{12}-q^2T_3c_{10}-q^3T_3c_9-q^3T_3c_8-q^4T_3c_7\\[-2 pt]
&\qquad-q^3T_3c_6-qT_3c_5-q^2T_3c_4-q^2T_3c_1-\!\!\sum_{i\in
\mathcal R(3,20)}\!\!\mu_{i,20}c_i.
\end{align*}
Now 3 is a descent of $t_{17}$, $t_{12}$, $t_8$, $t_7$ and $t_4$;
so
\begin{equation}\label{1st}\begin{split}
\ \ -q^2T_3c_{17}-q^3T_3c_{12}-q^3T_3c_8-q^4T_3c_7-q^2T_3c_4\hfill\\
\hfill=qc_{17}+q^2c_{12}+q^2c_8+q^3c_7+qc_4\ \
\end{split}
\end{equation}
and we see also that the sum \(\sum_{i\in \mathcal
R(3,20)}\mu_{i,20}c_i\) has no nonzero terms. Turning to the other
terms in the expression for \(c_{21}\), the coefficient of \(q\)
in the formula for \(c_{19}\) tells us that
\(\mu_{17,19}=\mu_{14,19}=\mu_{4,19}=1\), and thus
\begin{equation}\label{2nd}
-qT_3c_{19} =-q(qc_{19}+c_{17}+c_4)
\end{equation}
since $s_3t_{19}$ does not exist, and 3 is in $D(t_{17})$ and
$D(t_4)$ but not $D(t_{14})$. Similarly
\begin{align*}
-qT_3c_{18}&=-q(qc_{18}+c_{17})\\
-qT_3c_{15}&=-q(qc_{15}+c_{16})\\
-q^2T_3c_{14}&=-q^2(qc_{14}+c_{12})\displaybreak[0]\\
-q^2T_3c_{13}&=-q^2(qc_{13}+c_{12}+c_8)\displaybreak[0]\\
-q^2T_3c_{10}&=-q^2(qc_{10}+c_{11}+c_8)\displaybreak[0]\\
-q^3T_3c_9&=-q^3(qc_9+c_7+c_4)\displaybreak[0]\\
-q^3T_3c_6&=-q^3(qc_6+c_7)\\
-qT_3c_5&=-q(qc_5+c_4+c_3)\\
-q^2T_3c_1&=-q^2(qc_1+c_2),
\end{align*}
and we leave it to the reader to check that when these formulas
together with \eqref{1st} and \eqref{2nd} above are substituted
into our expression for $c_{21}$ the answer is as given
previously.

The above example is meant to illustrate a procedure that will
work for all Specht modules. Although it is clear enough that the
procedure will produce a basis \((\,c_{t} \mid t \in \STD(\lambda)\,)\)
such that the formulas in Equation \ref{wgformulas} hold, it is not
clear that the these formulas define a representation of \(\mathcal{H}\).
The proof that they do relies on Proposition~\ref{Howlettconj} below,
which is proved in~\cite{nguyen:wgideals2}. The algorithm has been
implemented using the computational algebra system MAGMA~\cite{Magma},
and in particular has been used in the case \(\lambda=(5,5,3,3)\) to
confirm the result of McLarnan and Warrington~\cite{mclwar} that
in this case 5 occurs as an edge-weight in the~\(W\!\)-graph.
\footnote{The magma files used can be obtained from
\texttt{http://www.maths.usyd.edu.au/u/bobh/magma/}, or from
\texttt{http://magma.maths.usyd.edu.au/magma/extra/}. It is planned
to include them in the next release of MAGMA.}

We now briefly indicate how to adapt the discussion of the
standard basis of \(S^\lambda\) given in \cite{math:heckeA} to
yield the formulas in~\eqref{stdbasis} above. It follows from
Corollary 3.4, Corollary 3.21 and Proposition 3.22 of
\cite{math:heckeA} that the Specht module \(S^{\lambda}\) (defined
immediately after Corollary 3.21) has a basis \((\,m_{t} \mid t
\in \STD(\lambda)\,)\) such that
\begin{equation*}
T_{i}m_{t} =
\begin{cases}
  m_{s_{i}t}  & \text{if \(i \in \SD(t)\),}\\
  qm_{s_{i}t} + (q - 1)m_{t} & \text{if \(i \in \SA(t)\),}\\
  qm_{t} & \text{if \(i \in \WA(t)\),}\\
  -m_{t} + \sum\limits_{s < t} a_{s,t}^{(i)}m_{s} & \text{if \(i \in \WD(t)\),}
\end{cases}
\end{equation*}
where the elements \(a_{s,t}^{(i)}\) are polynomials in \(q\).
Note that \cite{math:heckeA} employs the traditional definition of
\(\mathcal{H}_n\), so that to make the above formulas compatible
with our definitions we should replace \(q\) by \(q^{2}\) and
\(T_{i}\) by \(qT_{i}\). After this we use the automorphism of
\(\mathcal{H}_n\) given by \(T_{i} \rightarrow -T_{i}^{-1} =
-\overline{T_{i}}\) to define a new action, obtaining a module
that we call the dual Specht module. This gives
\begin{equation*}
-\overline{T_{i}}m_{t} =
\begin{cases}
  q^{-1}m_{s_{i}t}  & \text{if \(i \in \SD(t)\),}\\
  qm_{s_{i}t} + (q - q^{-1})m_{t} & \text{if \(i \in \SA(t)\),}\\
  qm_{t} & \text{if \(i \in \WA(t)\),}\\
  -q^{-1}m_{t} + q^{-1}\sum\limits_{s < t} a_{s,t}^{(i)}m_{s} & \text{if \(i\in\WD(t)\),}
\end{cases}
\end{equation*}
where now the \(a_{s,t}^{(i)}\) are polynomials in \(q^{2}\). We
apply these formulas for the module corresponding to the partition
\(\lambda'\) dual to \(\lambda\). This dualises again, swapping
ascents and descents, and giving a module that has a basis
\((\,m_{t} \mid t \in \STD(\lambda)\,)\) (not the same as the
\(m_t\)'s we started with) satisfying
\begin{equation*}
\overline{T_{i}}m_{t} =
\begin{cases}
 -q^{-1}m_{s_{i}t}  & \text{if \(i \in \SA(t)\),}\\
 -qm_{s_{i}t} - (q - q^{-1})m_{t} & \text{if \(i \in \SD(t)\),}\\
 -qm_{t} & \text{if \(i \in \WD(t)\),}\\
 q^{-1}m_{t} - q^{-1}\sum\limits_{s < t} a_{s,t}^{(i)}m_{s} & \text{if \(i \in \WA(t)\).}
\end{cases}
\end{equation*}
We now define \(b_t=(-q)^{l(t)}\overline{m_t}\). Applying the
involution \(v\mapsto\overline v\) to both sides of the above
formulas and multiplying through by \((-q)^{l(t)}\) yields
\eqref{stdbasis} above.

The following proposition was proved in the second author's PhD
thesis \cite[Theorem 6.3.4]{nguyen:thesis}. A shorter proof is presented
in \cite{nguyen:wgideals2}, a recently submitted followup to the present paper.
\begin{prop}\label{Howlettconj}
The elements \(r_{st}^{(i)}\) appearing in \eqref{stdbasis} are
polynomials in \(q\) with zero constant term.
\end{prop}
In fact Theorem 6.3.4 of \cite{nguyen:thesis} was stronger, saying
that the \(r_{st}^{(i)}\) are divisible by~\(q^2\). But the weaker
version is sufficient for our present needs.

Given that Proposition~\ref{Howlettconj} is true, it follows that
\(v_{\lambda}\) satisfies all the hypotheses in
Definition~\ref{wgphdetelt}, and is a \(W\!\)-graph determining
element relative to~\(J_\lambda\). According to the theory
presented in the next section, it follows that \(S^\lambda\) has a
\(W\!\)-graph basis \((\,c_{t} \mid t \in \STD(\lambda)\,)\) which
can be computed by means of the algorithm illustrated above.

\section{Constructing the \(W\!\)-graph from
a \(W\!\)-graph ideal}\label{section7}

We return now to the situation described in Section \ref{section5}
above, and let \(\mathscr I\) be a \(W\!\)-graph ideal with
respect to \(J\subseteq S\). By Definition~\ref{wgphdetelt} there
is an \(\mathcal{H}\)-module \(\mathscr{S}\) possessing an
\(\mathcal{A}\)-basis \(B=(\,b_w\mid w\in\mathscr I\,)\) on which
the generators of \(\mathcal{H}\) act via the formulas in
Definition~\ref{wgphdetelt}. Moreover, there is an
\(\mathcal{A}\)-semilinear involution \(v\mapsto\overline v\) on
\(\mathscr{S}\) satisfying \(\overline{b_1}=b_1\) and
\(\overline{hv} = \overline{h}\overline{v}\) for all \(h \in
\mathcal{H}\) and \(v \in \mathscr{S}\).

\begin{lem}\label{sinsumw1}
For each \(w\in \mathscr I\) there exist coefficients \(r_{y,w}
\in \mathcal{A}\), defined for \(y \in \mathscr I\) and \(y < w\),
such that \(\overline{b_{w}} - b_{w} = \sum r_{y,w}b_{y}\)
(summation over \(\{\,y\in\mathscr I\mid y<w\,\}\)).
\end{lem}
\begin{proof}
This is obvious when \(w=1\) since \(\overline{b_1} - b_1=0\) .
Proceeding inductively, suppose that \(w\in \mathscr I\) and
\(w\ne 1\), and choose \(s\in S\) such that \(w=su\) for some
\(u<w\). Then \(u\in \mathscr I\), and by the inductive hypothesis
there exist \(r_{z,u}\in\mathcal{A}\) with \(\overline{b_{u}} -
b_{u} = \sum_{\{z\in\mathscr I\mid z < u\}}r_{z,u}b_{z}\).
Moreover, \(s\in\SA(u)\), and so \(T_sb_u=b_{su}=b_w\). Thus
\begin{align*}
\overline{b_{w}} - b_{w} &=\overline{T_s}\,\overline{b_u}-T_sb_u\\
  &=(\overline{T_s}-T_s){b_u}+\overline{T_s}(\overline{b_u}-b_u)\\
&= (q^{-1}-q)b_u+\sum_{\substack{z < u\\z\in\mathscr
I}}r_{z,u}(T_s-(q-q^{-1}))b_{z}.
\end{align*}
Clearly \((q^{-1}-q)b_z\) is in the \(\mathscr A\)-module spanned
by \(\{y\,\in\mathscr I\mid y<w\,\}\)) whenever \(z\le u\), and so
it will suffice to show that \(T_sb_z\) is in this module whenever
\(z\in\mathscr I\) and \(z<u\). The formulas in
Definition~\ref{wgphdetelt} describe how to express \(T_sb_z\) as
an \(\mathcal{A}\)-linear combination of elements \(b_x\) for
\(x\in\mathscr I\), and our task is simply to check that every
\(x\) that occurs satisfies~\(x<w\).

The result is immediate if \(s\) is a weak descent of \(z\), since
in this case the only \(x\) that occurs is \(x=z\), and \(z<u<w\).
If \(s\) is a strong descent of \(z\) then \(x=z\) or \(x=sz\),
and in this case \(sz<z\). So again \(x\le z< u<w\), as required.

If \(s\) is a strong ascent of \(z\) then the only \(x\) that
occurs is \(x=sz\). Since \(z<u\) it follows from
Lemma~\ref{lifting1} that \(sz < su = w\), giving the required
result.

Finally, if \(s\) is a weak ascent of \(z\) then \(T_sb_{z}\) is a
linear combination of \(b_{z}\) and \(\{\,b_x\mid x\in\mathscr
I\text{ and }x<sz\,\}\). So either \(x=z<w\) or else \(x<sz <
su=w\) by  Lemma~\ref{lifting1}.
\end{proof}
Our aim is to construct a \(W\!\)-graph basis for \(\mathscr{S}\).
To do this we mimic the proof of Proposition~2 in
Lusztig~\cite{lusztig:leftcells}.
\begin{lem}\label{uniCbasis1}
The module \(\mathscr{S}\) has a unique \(\mathcal{A}\)-basis \(C
= (\,c_{w} \mid w\in \mathscr I\,)\) such that for all \(w \in
\mathscr I\) we have \(\overline{c_{w}} = c_{w}\) and
\begin{equation}\label{qpoly1}
b_{w} = c_{w} + q\sum_{y < w} q_{y,w}c_{y}
\end{equation}
for certain polynomials \(q_{y,w} \in \mathcal{A}^{+}\).
\end{lem}
\begin{proof}
Clearly \eqref{qpoly1} holds for \(w=1\) if and only if
\(c_{1}=b_{1}\), and defining \(c_{1}=b_{1}\) also ensures that
\(\overline{c_1}=c_1\), since \(\overline{b_1}=b_1\) is given.

Now suppose that \(w \ne 1\), and assume, inductively, that for
all \(y < w\) there exists a unique element \(c_{y}\in
\mathscr{S}\) such that \eqref{qpoly1} holds and
\(\overline{c_{y}}=c_{y}\). Then Lemma~\ref{sinsumw1} gives
\[
b_{w} - \overline{b_{w}} = \sum_{y < w} r_{y,w}c_{y}
\]
for some coefficients \(r_{y,w} \in \mathcal{A}\), and applying
the involution \(v\mapsto\overline v\) we see that
\(\overline{r_{y,w}}= -r_{y,w}\) for all \(y < w\), since
\eqref{qpoly1} and linear independence of the elements \(b_y\)
ensure linear independence of the~\(c_y\). So the coefficient of
\(q^0\) in \(r_{y,w}\) must be zero, and for \(n>0\) the
coefficient of \(q^{-n}\) must be the negative of the coefficient
of~\(q^n\). Hence \(r_{y,w} = qs_{y,w} - \overline{qs_{y,w}}\) for
a uniquely determined \(s_{y,w} \in \mathcal{A}^{+}\). Moreover,
\(q_{y,w}=s_{y,w}\) gives the unique solution to \(b_{w} = c_{w} +
q\sum_{y < w} q_{y,w}c_{y}\) with \(q_{y,w}\in\mathcal{A}^+\) and
\(\overline{c_{w}}=c_{w}\). So there is a unique element \(c_w\)
satisfying our requirements, and the induction is complete.
\end{proof}

Throughout the remainder of this section we let the elements
\(c_w\) and the polynomials \(q_{y,w}\) be defined so that the
conditions of Lemma~\ref{uniCbasis1} are satisfied. We also define
\(\mu_{y,w}\) to be the constant term of~\(q_{y,w}\).

\begin{thr}\label{main1}
Let \(s\in S\) and \(w \in \mathscr I\). Then
\[
T_sc_{w}=
\begin{cases}
-q^{-1}c_{w}&\text{if \(s\in\D(w)\),}\\
qc_{w}+\sum_{y\in \mathcal R(s,w)}\mu_{y,w}c_{y}
&\text{if \(s\in\WA(w)\),}\\
qc_{w}+c_{sw}+\sum_{y\in \mathcal R(s,w)}\mu_{y,w}c_{y} &\text{if
\(s\in\SA(w)\),}
\end{cases}
\]
where the set \(\mathcal R(s,w)\) consists of all \(y\in\mathscr I\)
such that \(y < w\) and \(s\in\D(y)\).
\end{thr}
\begin{proof}
Suppose first that \(w = 1\). If \(s\notin\mathscr I\) then either
\(s\in \WD(1)\) (if \(s\in J\)) or \(s\in \WA(1)\) (if \(s\notin
J\)), and since \(c_{1} = b_{1}\) it follows from the formulas in
Definition~\ref{wgphdetelt} that
\begin{equation*}
T_sc_{1} = \begin{cases} -q^{-1}c_{1} & \text{ if \(s\in\WD(1)\)}\\
                         qc_{1} & \text{ if \(s\in\WA(1)\)}.
                           \end{cases}
\end{equation*}
Since the set \(\mathcal R(s,1)\) is obviously empty, the formulas
in the statement of the theorem hold in these two cases. If
\(s\in\mathscr I\) then clearly \(s\in\SA(1)\) since \(s1<1\) is
impossible, and in this case Definition~\ref{wgphdetelt} gives
\(T_sb_1=b_s\). So
\[
b_{s}-\overline{b_{s}}=T_sc_1-\overline{T_sc_1}=(T_s-\overline{T_s})c_1=(q-q^{-1})c_{1}.
\]
Thus \(q_{1,s}=1\) and \eqref{qpoly1} becomes \(b_s=c_s+qc_1\),
giving
\[
T_sc_1=b_s=qc_1+c_s,
\]
as required.

Proceeding by induction, suppose now that $w > 1$, and consider
first the case that \(s\in\SD(w)\). Then \(y = sw < w\), and
\(s\in\SA(y)\); so the inductive hypothesis gives
\[
T_sc_y=qc_y+c_w+\!\!\!\sum_{x \in \mathcal
R(s,y)}\!\!\mu_{x,y}c_{x},
\]
which can be rewritten as
\[
c_{w} = (T_s-q)c_{y} - \!\!\!\sum_{x \in \mathcal R(s,y)}
\!\!\mu_{x,y}c_{x}.
\]
Since \(T_s(T_s-q)=-q^{-1}(T_s-q)\) and \(T_sc_{x} =
-q^{-1}c_{x}\) for all \(x\in \mathcal R(s,y)\) (by the inductive
hypothesis), it follows that
\[
T_sc_{w} = -q^{-1}c_{w}.
\]
as required.

Now consider the case that \(s\in\WD(w)\).
Definition~\ref{wgphdetelt} gives
\[(T_s+q^{-1})b_{w}=0,\]
and so by \eqref{qpoly1},
\begin{equation}\label{WD1}
(T_s+q^{-1})c_{w}=-q\sum_{y < w}q_{y,w}(T_s+q^{-1})c_{y}.
\end{equation}
If \(y<w\) then \((T_s+q^{-1})c_{y}=0\) if \(s\in\D(y)\), while if
\(s\notin\D(y)\) then
\[
(T_s+q^{-1})c_{y}=(q+q^{-1})c_y+v_y
\]
for some \(v_y\) in \(\mathscr{S}_{\!s}^-\), the subspace
of~\(\mathscr{S}\) spanned by \(\{\,c_{x}\mid s\in\D(x)\,\}\).
Hence
\[
(T_s+q^{-1})c_{w}=\Bigl(-\sum_{y\in Y}q_{y,w}(q^2+1)c_{y}\Bigr)+v
\]
for some \(v\in \mathscr{S}_{\!s}^-\), where \(Y=\{\,y\mid
y<w\text{ and }s\notin\D(y)\,\}\). Since the map
\(v\mapsto\overline v\) fixes \((T_s+q^{-1})c_{w}\) it follows
that when \((T_s+q^{-1})c_{w}\) is expressed as a linear
combination of the the basis elements \(c_y\), all the
coefficients are fixed. Hence
\[
\overline{(q^2+1)q_{y,w}}=(q^2+1)q_{y,w}
\]
for all~\(y<w\) such that \(s\notin\D(y)\). But since
\((q^2+1)q_{y,w}\) is a polynomial in~\(q\) this forces it to be a
constant, and hence forces \(q_{y,w}=0\). So all the terms on the
right-hand side of \eqref{WD1} disappear, and
\[
T_sc_{w}=-q^{-1}c_{w}
\]
as required. Note that this argument has shown that for all \(s\)
such that \(s\in\WD(w)\), the right-hand side of \eqref{qpoly1}
involves only elements \(c_y\) such that \(s\in\D(y)\).

Suppose next that \(s\in\SA(w)\), so that \(w<sw\in \mathscr I\).
By Lemma~\ref{uniCbasis1} and Definition~\ref{wgphdetelt} we have
\begin{align*}
T_sc_{w} &=T_sb_{w}-q\sum_{y< w}q_{y,w}T_sc_{y}\\
          &=b_{sw}-q\sum_{y< w}q_{y,w}T_sc_{y}\\
          &=c_{sw}+q\sum_{y< sw}q_{y,sw}c_{y}
            -q\sum_{y< w}q_{y,w}T_sc_{y}.
\end{align*}
Applying the inductive hypothesis to evaluate the \(T_sc_{y}\) in
the second sum gives
\begin{multline*}
(T_s-q)c_{w} =c_{sw}-qc_{w}+q\sum_{y<sw}q_{y,sw}c_{y}
+\!\!\!\sum_{y\in \mathcal R(s,w)}\!\!q_{y,w}c_{y}\\[-2 pt]
-q\!\!\!\sum_{y\in \mathcal R'(s,w)}\!\!
q_{y,w}\Bigl(qc_{y}+c_{sy}+\!\!\!\sum_{x\in \mathcal R(s,y)}
\!\!\mu_{x,y}c_{x}\Bigr)
\end{multline*}
where \(c_{sy}\) is to be interpreted as zero if \(s\in\WA(y)\),
and we have written \(\mathcal R'(s,w)\) for the set of all \(y <
w\) such that \(s\notin\D(y)\). Now since there are no negative
powers of \(q\) appearing in any of the coefficients on the right
hand side, but \(\overline{(T_s-q)c_{w}}=(T_s-q)c_{w}\), we deduce
that all the coefficients must simply be integers, and the
positive powers of \(q\) must cancel out. So
\[
(T_s-q)c_{w}=c_{sw}+\!\!\!\sum_{y\in \mathcal
R(s,w)}\!\!\mu_{y,w}c_{y},
\]
where \(\mu_{y,w}\) is the constant term of \(q_{y,w}\), as
required.

As a by-product of the above calculations we have shown that
\begin{equation}\label{recformula}
\begin{split}
-qc_{w}+q\sum_{y< sw}q_{y,sw}c_{y}
&+\!\!\!\!\sum_{y\in \mathcal R(s,w)}\!\!\!(q_{y,w}-\mu_{y,w})c_{y}\\
&=q\!\!\!\sum_{y\in \mathcal R'(s,w)}\!\!
q_{y,w}\Bigl(qc_{y}+c_{sy}+\!\!\!\sum_{x\in \mathcal R(s,y)}
\!\!\mu_{x,y}c_{x}\Bigr),
\end{split}
\end{equation}
whenever \(w<sw\in \mathscr I\). We shall return to this below,
and use it to obtain a recursive formula for the
polynomials~\(q_{y,w}\).

Finally, suppose that \(s \in\WA(w)\). By (i) of
Definition~\ref{wgphdetelt} this gives
\[
(T_s-q)b_{w}=-\sum_{y < sw}r_{y,w}^sb_{y},
\]
for some \(r_{y,w}^s \in q\mathcal{A}^{+}\), so that by
Lemma~\ref{uniCbasis1}
\[
(T_s-q)c_{w}+q\sum_{y < w}q_{y,w}(T_s-q)c_{y} =
 -\sum_{y<sw}r_{y,w}^s\Bigl(c_{y}+q\sum_{x < y}q_{x,y}c_{x}\Bigr).
\]
Hence \((T_s-q)c_{w}\) is equal to
\begin{equation}\label{tk-qc}
\begin{split}\quad
\sum_{y\in\mathcal R(s,w)}\!\!\!q_{y,w}(q^2+1)c_{y}
-\!\!\!\!\sum_{y\in\mathcal R'(s,w)}\!\!\!qq_{y,w}
\Bigl(c_{sy}&+\!\!\!\!\sum_{x\in\mathcal R(s,y)}\!\!\!
\mu_{x,y}c_{x}\Bigr)\\ &-\sum_{y <
sw}r_{y,w}^s\Bigl(c_{y}+q\sum_{x < y}q_{x,y}c_{x}\Bigr),
\end{split}
\end{equation}
where again \(c_{sy}\) is interpreted as 0 if \(s\in\WA(y)\).
Since \(\overline{(T_s-q)c_{w}} = (T_s-q)c_{w}\) it follows again
that all terms involving positive powers of \(q\) must cancel out;
this includes all of \(\sum_{y < sw}r_{y,w}^s(c_{y}+q\sum_{x <
y}q_{x,y}c_{x})\) since \(r_{y,w}^s\in q\mathcal{A}^+\). Hence
\[
(T_s-q)c_{w} = \!\!\!\sum_{y\in \mathcal
R(s,w)}\!\!\mu_{y,w}c_{y},
\]
where \(\mu_{y,w}\) is the constant term of \(q_{y,w}\), as
required.
\end{proof}

Returning now to \eqref{recformula}, which holds whenever
\(w<sw\in \mathscr I\), we proceed to derive the promised
recursive formula for the polynomials~\(q_{y,w}\).

Observe first that \(c_{w}\) does not occur on the right hand side
of \eqref{recformula} or in the last sum on the left hand side;
hence it follows that \( q_{w,sw}=1\). Next, examining the
coefficients of \(c_{z}\) when \(z\in\mathcal{R}'(s,w)=\{\,z<w\mid
s\notin\D(z)\,\}\) gives \(q_{z,sw}=qq_{z,w}\) in this case. (Note
that when \(z\ne w\) and \(s\notin\D(z)\) the conditions \(z< w\)
and \(z< sw\) are equivalent, by Lemma \ref{lifting1}.) Finally,
suppose that \(z<sw\) and \(s\in\D(z)\). If \(z\not<w\) then
\(z=sy\) for some \(y\in\mathcal{R}'(s,w)\), and the coefficient
of \(c_z\) on the right hand side of \eqref{recformula} is
\(qq_{sz,w}\), while on the left hand side it is \(qq_{z,sw}\).
Thus \(q_{sz,w}=q_{z,sw}\) in this case. If \(z<w\) and
\(s\in\SD(z)\) then \(sz\in\mathcal{R}'(s,w)\), and we see that
\(c_z\) occurs on the right hand side of \eqref{recformula} as
\(c_{sy}\) when \(y=sz\), and also occurs in the sums
\(\sum_{x\in\mathcal R(s,y)}\mu_{x,y}c_{x}\) for those
\(y\in\mathcal{R}'(s,w)\) such that \(z<y\). Thus the coefficient
of \(c_z\) on the right hand side of \eqref{recformula} is
\(qq_{sz,w}+q\sum_y\mu_{z,y}q_{y,w}\), where the sum is over all \(y\in
\mathscr I\) such that \(z< y < w\) and \(s\notin D(y)\). On the
left hand side of \eqref{recformula} the coefficient of \(c_z\) is
\(qq_{z,sw}+(q_{z,w}-\mu_{z,w})\). Hence
\begin{equation}
q_{z,sw} = -q^{-1}(q_{z,w}-\mu_{z,w})+q_{sz,w}
+\sum_y\mu_{z,y}q_{y,w}
\end{equation}
where the sum is over all \(y\in \mathscr I\) such that \(z< y <
w\) and \(s\notin D(y)\). If \(z<w\) and \(s\in\WD(z)\) then we
obtain the same formula without the \(q_{sz,w}\) term.

We have proved the following result.
\begin{cor}\label{recursion}
Suppose that \(w<sw\in \mathscr I\) and \(y<sw\). If \(y=w\) then
\(q_{y,sw}=1\), and if \(y\ne w\) we have the following formulas:
\begin{itemize}
\item[(i)] \(q_{y,sw}=qq_{y,w}\) \ if \(s\in\A(y)\), \item[(ii)]
\(q_{y,sw} = -q^{-1}(q_{y,w}-\mu_{y,w})+q_{sy,w}
+\sum_x\mu_{y,x}q_{x,w}\) \ if \(s\in\SD(y)\), \item[(iii)]
\(q_{y,sw} = -q^{-1}(q_{y,w}-\mu_{y,w})+\sum_x\mu_{y,x}q_{x,w}\) \
if \(s\in\WD(y)\),
\end{itemize}
where \(q_{y,w}\) and \(\mu_{y,w}\) are regarded as \(0\) if
\(y\not<w\), and in \textup{(ii)} and \textup{(iii)} the sums
extend over all \(x\in \mathscr I\) such that \(y<x<w\) and
\(s\notin\D(x)\).
\end{cor}
The following result follows easily from Corollary~\ref{recursion}
by induction on~\(l(w)-l(y)\).
\begin{prop}\label{qdegree}
Let \(y<w\in\mathscr{I}\!\). Then the degree of \(q_{y,w}\) is at
most~\(l(w)-l(y)-1\).
\end{prop}
Now let \(\mu\colon C \times C \rightarrow \mathbb{Z}\) be given
by
\begin{equation}\label{W-graphCoefficients}
\mu(c_{y},c_{w}) =
                     \begin{cases}
                         \mu_{y,w} &\text{if $y < w$}\\
                         \mu_{w,y} &\text{if $w < y$}\\
                         0 &\text{otherwise},
                     \end{cases}
\end{equation}
and let \(\tau\) from \(C\) to the power set of~\(S\) be given by
\(\tau(c_{w}) = \D(w)\) for all~\(y\in\mathscr{I}\!\).
\begin{thr}\label{main-wg}
The triple \((C, \mu, \tau)\) is a \(W\!\)-graph.
\end{thr}
\begin{proof}
In view of Theorem \ref{main1} it suffices to show that for all
\(w\in \mathscr I\) and \(s\in S\), if \(s\in \WA(w)\) then the
set \(\{\,y\in \mathscr I\mid s\in\D(y)\text{ and }\mu(c_y,c_w)\ne
0\,\}\) contains no elements \(y>w\), while if \(s\in\SA(w)\) then
the only such element is \(sw\), and \(\mu(c_{sw},c_w)=1\).

Accordingly, suppose that \(w<y\in \mathscr I\) with
\(\mu_{w,y}\ne0\), and suppose that \(s\in\D(y)\cap \A(w)\). As
noted in the proof of Theorem~\ref{main1}, if \(s\in\WD(y)\) then
\(q_{z,y}=0\) for all \(z<y\) with \(s\in \A(y)\); in particular,
\(q_{w,y}=0\), contradicting~\(\mu_{w,y}\ne0\). Hence
\(s\in\SD(y)\). Now define \(x=sy\), so that \(x<sx\in \mathscr
I\), and observe by Corollary~\ref{recursion}~(i) that
\(q_{w,sx}=qq_{w,x}\) if \(w\ne x\). Since this contradicts
\(\mu_{w,y}\ne 0\) we conclude that \(w=x\), and \(q_{w,sx}=1\),
by Corollary~\ref{recursion}. So \(y=sw\) and \(\mu_{w,sw}=1\), as
required.
\end{proof}

\begin{prop}\label{inversepolys}
Let the bases \(B=(b_w\mid w\in\mathscr I)\) and 
\(C=(c_w\mid w\in\mathscr I)\) be as in Lemma~\ref{uniCbasis1} above.
Then there exist polynomials \(p_{y,w}\in\mathcal{A}^+\) such
that \(c_{w} = b_{w} - q\sum_{y < w}p_{y,w}b_{y}\) for all
\(w\in\mathscr I\), and the constant term of \(p_{y,w}\)
is~\(\mu_{y,w}\).
\end{prop}

\begin{proof}
It follows readily from Equation \ref{qpoly1} that the required
polynomials \(p_{y,w}\) are given recursively by
\begin{equation}\label{pdashpoly1}
p_{y,w} = q_{y,w} - \sum_{y<x<w}qp_{y,x}q_{x,w} \quad \text{if $y
< w$},
\end{equation}
whence the constant term of \(p_{y,w}\) equals
that of \(q_{y,w}\) .
\end{proof}
For our final theoretical result of this section, we show that if
\(\mathscr I\) is a \(W\!\)-graph ideal that is generated by a
single (\(W\!\)-graph determining) element, then in Part~(i) of
Definition~\ref{wgphdetelt}, in the case \(s\in\WA_J(w)\), the sum
\(\sum_{y \in \mathscr I\!\!,\,y < sw} r^s_{y,w}b_{y}\)
can be replaced by the simpler \(\sum_{y \in \mathscr I\!\!,\,y < w} r^s_{y,w}b_{y}\).

\begin{lem}\label{bruhat-technical}
Suppose that \(x,\,y,\,v,\,w\in W\) and \(s\in S\) satisfy
\begin{enumerate}
\item \(xy=vw\) and \(l(xy)=l(x)+l(y)=l(v)+l(w)\),
\item \(sw>w\) and \(vs>v\),
\item \(y\le sw\).
\end{enumerate}
Then \(y\le w\).
\end {lem}
\begin{proof} Assume that \(x,\,y,\,v,\,w\) and \(s\) satisfy the stated
hypotheses. If \(sy>y\) then the desired conclusion follows immediately from the
hypotheses \(y\le sw\) and \(sw>w\), by Lemma~\ref{lifting1}. So we may assume that
\(sy<y\). With this extra hypothesis, we use induction on \(l(w)\) to prove the
result.

If \(l(w)=0\) then the hypothesis (3) becomes \(y\le s\), and since \(sy<y\) it
follows that \(y=s\). So \(l(x)+l(y)=l(v)+l(w)\) becomes \(l(x)=l(v)-1\),
and \(xy=vw\) becomes \(xs=v\), which together contradict the hypothesis \(vs>v\).
So the result is vacuously true in this case.

Now suppose that \(l(w)>0\) and that the result holds in all cases corresponding to
shorter~\(w\). Choose \(r\in S\) such that \(w'=wr<w\). Note that since
\[
l(w')+1=l(w)<l(sw)=l(sw'r)\le l(sw')+1
\]
and also
\[
l(sw')\le l(w')+1=l(w)<l(sw)=l(sw'r)
\]
it follows that \(w'<sw'\) and \(sw'<sw'r\).

Suppose first that \(yr>y\). By hypothesis~(1),
\[
l(xyr)=l(vwr)=l(vw')\le l(v)+l(w')=l(v)+l(w)-1=l(xy)-1,
\]
and so \(xyr=x'y\) for some \(x'\) with \(l(x')=l(x)-1\), by the Exchange Condition.
Moreover, since \(sw'<sw'r=sw\) (proved above) and \(y<yr\), it follows from Lemma~\ref{lifting1}
and the hypothesis \(y<sw\) that \(y\le sw'\). So now we have
\begin{itemize}
\item[\((1')\)] \(x'y=vw'\) and \(l(x'y)=l(x')+l(y)=l(v)+l(w')\),
\item[\((2')\)] \(sw'>w'\) and \(vs>v\),
\item[\((3')\)] \(y\le sw'\),
\end{itemize}
and since \(l(w')<l(w)\) the inductive hypothesis gives \(y\le w'\). But \(w'<w\);
so \(y\le w\) in this case.

It remains to consider the case \(yr<y\). Put \(y'=yr\), and observe that
\[
l(xy')\le l(x)+l(y')=l(x)+l(y)-1=l(xy)-1\le l(xyr)=l(xy'),
\]
so that \(l(xy')=l(x)+l(y')\). The same argument gives \(l(vw')=l(v)+l(w')\).
And since \(sw'<sw'r=sw\) and \(y'<y'r=y\), it follows from the hypothesis
\(y<sw\) and Lemma~\ref{lifting1} that \(y'<sw'\). So now we have
\begin{itemize}
\item[\((1'')\)] \(x'y=vw'\) and \(l(xy')=l(x)+l(y')=l(v)+l(w')\),
\item[\((2'')\)] \(sw'>w'\) and \(vs>v\),
\item[\((3'')\)] \(y'\le sw'\),
\end{itemize}
and since \(l(w')<l(w)\) the inductive hypothesis gives \(y'\le w'\). Since
\(y'<y'r=y\) and \(w'<w'r=w\), this yields \(y<w\), by Lemma~\ref{lifting1}.
\end{proof}

\begin{prop}\label{simplerdef} Suppose that \(u\in W\) and
\(\mathscr I=\{\,w\in W\mid w\le_L u\,\}\) is a \(W\!\)-graph ideal
with respect to~\(J\). With all the notation as in Definition~\ref{wgphdetelt},
if \(w\in\mathscr I\) and \(s\in\WA_J(w)\), then every \(y\in\mathscr I\)
with \(y<sw\) satisfies \(y\le w\).
\end{prop}
\begin{proof}
Suppose that \(w\in\mathscr I\) and \(s\in\WA_J(w)\), and that \(y\in\mathscr I\)
with \(y<sw\). Since \(w\) and \(y\) are in \(\mathscr I\) they are both suffixes
of~\(u\), and so there exist \(x,\,v\in W\) with \(u=xy=vw\) and
\(l(u)=l(x)+l(y)=l(v)+l(w)\). If \(v'=vs<v\) then \(u=(v's)w=v'(sw)\), showing
that \(sw\le_Lu\) since 
\[
l(u)=l(v)+l(w)=l(v')+1+l(w)=l(v')+l(sw).
\]
Since this contradicts the assumption that \(s\in\WA_J(w)\), it follows that
\(vs>v\), and hence all the hypotheses of Lemma~\ref{bruhat-technical} are
satisfied. So \(y\le w\), as required.
\end{proof}

The recursive nature of Corollary~\ref{recursion} makes
it relatively straightforward to implement calculation
of the polynomials \(q_{y,w}\) (and hence the \(W\!\)-graph
edge-weights~\(\mu_{y,w}\)) using a computational algebra program.
We outline one possible way to do this.

Assume that the elements of \(\mathscr I\) are listed as
\(w_1,w_2,\ldots,w_d\), where $i\le j$ implies that \(l(w_i)\le l(w_j)\),
and let  \(S=\{s_1,s_2,\ldots s_n\}\). The input to the process
is an array \texttt{tab} such that \texttt{tab[i,j] = j} if
\(s_i\) is a weak ascent of~\(w_j\) and \texttt{tab[i,j] = -j} if
\(s_i\) is a weak descent of~\(w_j\), while \texttt{tab[i,j] = k}
if \(s_i\) is a strong ascent or strong descent of \(w_j\)
and \(s_iw_j=w_k\). It is convenient to precompute another
array \texttt{descents} such that
\[
\texttt{descents[j] = \{ i | tab[i,j] < j \}}.
\]
We can now define a function \texttt{Q} such that \texttt{Q(j,k)}
returns the polynomial \(q_{y,z}\) if \(y=w_j<z=w_k\),
and returns~\(0\) otherwise.

If \texttt{j \(\ge\) k} then \texttt{Q(j,k)} immediately returns~0.
Otherwise the set \texttt{descents[k]} is searched for an \texttt{s} with
\texttt{tab[s,k] = m > 0}; note that since \texttt{m < k} the value of
\texttt{Q(j,m)} can be used in the calculation of \texttt{Q(j,k)}.
By (i) of Corollary~\ref{recursion}, \texttt{Q(j,k)} can be
set equal to \(q*{}\)\texttt{Q(j,m)} if \texttt{s} is not in
\texttt{descents[j]}. If \texttt{s} is in \texttt{descents[j]}
then (ii) of Corollary~\ref{recursion} is applicable if
\texttt{tab[s,j] > 0}, while (iii) is applicable if
\texttt{tab[s,j] < 0}. Interpreting \texttt{Q(tab[s,j],m)} as zero
in this latter case, the formula for \texttt{Q(j,k)} becomes
\[
\texttt{Q(j,k) = ((mu(j,m) - Q(j,m))}/q\texttt{) + Q(tab[s,j],m) + Sum}
\]
where \texttt{mu(j,m)} is the constant term of \texttt{Q(j,m)}, and
\texttt{Sum} denotes the sum of the values \texttt{mu(j,i) \(*\) Q(i,m)}
for \texttt{i} in the range \texttt{j < i < m}.

\section{The Kazhdan-Lusztig and Deodhar constructions}\label{klparrel}

Since every \(u\in W\) occurs as a suffix of the longest element
\(w_S\), the ideal of \((W,\le_L)\) generated by \(w_S\) is the
whole of~\(W\). We seek to show that \(w_S\) is a \(W\!\)-graph
determining element, or, equivalently, that \(W\) is a
\(W\!\)-graph ideal. We are forced to let \(J=\emptyset\) so that
the requirement \(W\subseteq D_J\) is satisfied, and this means
that the sets \(\WA_J(w)\) and \(\WD_J(w)\) are empty for all
\(w\in W\). Hence to show that \(w_S\) is a \(W\!\)-graph
determining element we need to produce an \(\mathcal{H}\)-module
with an \(\mathcal A\)-basis \((\,b_w\mid w\in W\,)\) such that
for all \(s\in S\) and \(w\in W\),
\[
T_sb_w=
\begin{cases}
  b_{sw}&\text{if \(sw>w\)}\\
  b_{sw}+(q-q^{-1})b_w&\text{if \(sw<w\).}
\end{cases}
\]
The module must also admit an \(\mathcal A\)-semilinear involution
such that \(\overline{T_wb_1}=\overline{T_w}b_1\) for all \(w\in
W\). Since these conditions are obviously satisfied if we put
\(b_w=T_w\) for all \(w\in W\), the required module is the left
regular module \(\mathcal{H}\). Thus our construction in
Section~\ref{section7} will produce a \(W\!\)-graph basis
of~\(\mathcal{H}\), and combining Propositions~\ref{qdegree}
and~\ref{inversepolys} yields the following result.
\begin{prop}\label{kltheorem}
The Hecke algebra \(\mathcal H\) has a \(W\!\)-graph basis
\((c_w\mid w\in W)\) such that \(\overline{c_w}=c_w\) and
\(c_w=T_w-\sum_{y<w}qp_{y,w}T_y\) for all \(w\in W\), where
\(p_{y,w}\) is a polynomial of degree at
most \(l(w)-l(y)-1\) and the W-graph edge-weight \(\mu_{y,w}\) is
the constant term of~\(p_{y,w}\).
\end{prop}
Converting the traditional version of \(\mathcal H\) as used
in~\cite{kazlus:coxhecke} to our version requires replacing
\(q\) by \(q^2\), after which the \(T_w\) of \cite{kazlus:coxhecke}
becomes \(q^{l(w)}T_w\) in our context. So the formula in
\cite[Theorem 1.1]{kazlus:coxhecke}, when converted to our context,
becomes \(C_w=\sum_{y\le w}(-1)^{l(w)-l(y)}q^{l(w)-2l(y)}
P_{y,w}^*(q^{l(y)}T_y)\), where \(P_{y,w}^*\) is obtained from
the Kazhdan-Lusztig polynomial \(P_{y,w}\) by replacing \(q\)
by~\(q^{-2}\). Since \(P_{w,w}=1\) the coefficient of \(T_w\)
on the right hand side of this expression is~1, and since \(P_{y,w}\)
is a polynomial of degree at most \(\frac12(l(w)-l(y)-1)\) when
\(y<w\) we see that the coefficient of \(T_y\), namely
\((-q)^{l(w)-l(y)} P_{y,w}^*\), is a polynomial in
\(q\) with zero constant term. Since also \(\overline{C_w}=C_w\),
the uniqueness  part of Lemma~\ref{uniCbasis1} guarantees that 
\(C_w=c_w\), from which we can deduce a simple relationship between
our polynomials \(p_{y,w}\) and the Kazhdan-Lusztig polynomials.
\begin{prop}\label{polyrel}The polynomials \(p_{y,w}\) appearing
in Proposition~\ref{kltheorem} are related to the Kazhdan-Lusztig
polynomials \(P_{y,w}\) via
\begin{equation}\label{KLPrelation}
p_{y,w} = (-q)^{l(w) - l(y) - 1}P_{y,w}^*.
\end{equation}
where \(P_{y,w}^*\) is obtained from \(P_{y,w}\) by replacing \(q\)
by~\(q^{-2}\). In particular, the coefficient of 
\(q^{\frac12(l(w)-l(y)-1)}\) in \(P_{y,w}\) is
\((-1)^{l(w)-l(y)-1}\mu_{y,w}\).
\end{prop}

Note that Kazhdan and Lusztig show that \(\mu_{y,w}\ne 0\)
only if \(l(w)-l(y)-1\) is~even.

Turning now to Deodhar's construction, let \(J\) be an arbitrary
subset of~\(S\) and let \(d_J\) be the longest element of
\(D_{J}\) (which is the shortest element of \(w_SW_J\)). An
element \(u\in W\) is a suffix of \(d_J\) if and only if \(u\in
D_J\), and so the ideal \(\mathscr{I}\) of \((W,\le_L)\) generated
by \(d_J\) coincides with~\(D_J\). Clearly
\(\Pos(\mathscr{I})=J\). We shall show that \(\mathscr{I}=D_J\) is
a \(W\!\)-graph ideal with respect to~\(J\), and also that it is a
\(W\!\)-graph ideal with respect to~\(\emptyset\). We consider the
latter case first.

Since \(D_\emptyset=W\), it follows from the definitions in
Section~\ref{section5} that if \(w\in \mathscr{I}\) then
\(\SA(w)=\{\,s\in S\mid sw>w\text{ and }sw\in D_J\,\}\) and
\(\SD(w)=\{\,s\in S\mid sw<w\,\}\), while
\(\WD_\emptyset(w)=\{\,s\in S\mid sw\notin D_\emptyset\,\}=\emptyset\)
and
\[
\WA_\emptyset(w)=\{\,s\in S\mid sw\in D_\emptyset\setminus D_J\,\}\\
=\{\,s\in S\mid sw=wt\text{ for some }t\in J\,\}
\]
by Lemma~\ref{deo1}. We proceed to construct an
\(\mathcal{H}\)-module \(\mathscr{S}\) satisfying the requirements
of~Definition~\ref{wgphdetelt}. (Our module \(\mathscr{S}\) is
essentially the module \(M^J\) in~\cite{deo:paraKL}, in the case
\(u=q\), the only differences being due to our non-traditional
definition of~\(\mathcal{H}\).)

Let \(\mathcal{H}_J\) be the Hecke algebra associated with the
Coxeter system \((W_{J},J)\), and recall that \(\mathcal{H}_J\)
can be identified with the subalgebra of \(\mathcal{H}\) spanned
by \(\{\,T_u\mid u\in W_J\,\}\). There is an \(\mathcal
A\)-algebra homomorphism \(\psi\colon\mathcal{H}_J \to
\mathcal{A}\) such that \(\psi(T_u)=q^{l(u)}\) for all \(u\in
W_J\), and this can be used to give \(\mathcal{A}\) the structure
of an \(\mathcal{H}_J\)-module, which we denote by
\(\mathcal{A}_\psi\). Since \(\mathcal{H}\) is obviously an
\((\mathcal{H},\mathcal{H}_J)\)-bimodule, the tensor product
\(\mathscr{S}_{\!\psi} = \mathcal{H}\otimes_{\mathcal{H}_J}
\mathcal{A}_\psi\) is a (left) \(\mathcal{H}\)-module, and it is
straightforward to show that it is \(\mathcal A\)-free with basis
\(B=(\,b_w\mid w\in D_J\,)\) defined by \(b_w=T_w\otimes 1\) for
all~\(w\in D_J\).

Let \(w\in D_J\) and \(s\in S\). If \(s\in \SA(w)\) then
\(l(sw)>l(w)\), and so
\[
T_sb_w=T_s(T_w\otimes 1)=(T_sT_w)\otimes 1=T_{sw}\otimes 1=b_{sw}
\]
since \(sw\in D_J\). If \(s\in \SD(w)\) then \(l(sw)<l(w)\), and
so
\[
T_sb_w=(T_sT_w)\otimes 1=(T_{sw}+(q-q^{-1})T_w)\otimes
1=b_{sw}+(q-q^{-1})b_w
\]
since again \(sw\in D_J\). There are no weak descents, and if
\(s\in\WA_\emptyset(w)\) then there is a \(t\in J\) with
\(sw=wt\), and we find that
\[
T_sb_w=(T_sT_w)\otimes 1=(T_wT_t)\otimes
1=T_w\otimes\psi(T_t)=qb_w.
\]
So the action of the generators \(\{\,T_s\mid s\in S\,\}\) on the
basis \(B\) is in accordance with the requirements of
Definition~\ref{wgphdetelt}~(i) (with all the polynomials
\(r_{y,w}^s\) being zero), and it only remains to check that
\(\mathscr{S}_{\!\psi}\) admits an \(\mathcal{A}\)-semilinear
involution satisfying the requirements of
Definition~\ref{wgphdetelt}~(ii). We include a proof here for the
sake of completeness, although the result is proved
in~\cite{deo:paraKL}.

We show that the unique \(\mathcal{A}\)-semilinear map
\(\mathscr{S}_{\!\psi}\to\mathscr{S}_{\!\psi}\) satisfying
\(\overline{b_w}=\overline{T_w}\otimes 1\) for all \(w\in D_J\)
has the required properties. Note first that
\(\psi(\overline{T_u})=\psi(T_u)^{-1}=\overline{\psi(T_u)}\) for
all \(u\in W_J\). Now if \(x\in W\) is arbitrary then we may write
\(x=wu\) for some \(w\in D_J\) and some \(u\in W_J\), and we find
that
\begin{equation*}
\begin{split}
\overline{T_x\otimes
1}=&\overline{T_wT_u\otimes1}=\overline{T_w\otimes\psi(T_u)}
=\overline{\psi(T_u)(T_w\otimes 1)}=\overline{\psi(T_u)}(\overline{T_w\otimes 1})\\
&=\psi(\overline{T_u})(\overline{T_w}\otimes1)=
\overline{T_w}\otimes\psi(\overline{T_u})
=\overline{T_w}\,\overline{T_u}\otimes1
=\overline{T_wT_u}\otimes1=\overline{T_x}\otimes 1.
\end{split}
\end{equation*}
Hence \(\overline{k\otimes 1}=\overline{k}\otimes1\) for all
\(k\in\mathcal{H}\), and so
\[
\overline{h(k\otimes
1)}=\overline{(hk)\otimes1}=\overline{hk}\otimes1=(\overline{h}\,\overline{k})\otimes1
=\overline{h}(\overline{k\otimes 1})
\]
for all \(h,\,k\in\mathcal{H}\). So
\(\overline{h\alpha}=\overline{h}\overline{\alpha}\) for all
\(h\in\mathcal{H}\) and \(\alpha\in \mathscr{S}_{\!\psi}\), as
required.

Since the requirements of Definition~\ref{wgphdetelt} have all
been met, the construction in Section~\ref{section7} above
produces a \(W\!\)-graph basis in the
module~\(\mathscr{S}_{\!\psi}\). This basis corresponds to the
basis of \(M^J\) in Proposition 3.2 (iii) of \cite{deo:paraKL} (in
the case \(u=q\)). Deodhar's polynomials \(P_{y,w}^J\) and our
polynomials are related by the obvious modification of~\eqref{KLPrelation}
above.

\begin{prop}\label{deo-case1}
The \(\mathcal H\)-module \(\mathscr{S}_{\!\psi}\) has a \(W\!\)-graph
basis \((\,c_w\mid w\in D_J\,)\) such that \(\overline{c_w}=c_w\) and
\(c_w=b_w-\sum_{y<w}qp_{y,w}^Jb_y\) for all \(w\in W\), where
\(p_{y,w}^J\) is a polynomial of degree at
most \(l(w)-l(y)-1\) and the W-graph edge-weight \(\mu_{y,w}\) is
the constant term of~\(p_{y,w}^J\). The polynomials \(p_{y,w}^J\)
are related to Deodhar's polynomials \(P_{y,w}^J\) via
\begin{equation}\label{Deorelation}
p_{y,w}^J = (-q)^{l(w) - l(y) - 1}P_{y,w}^*.
\end{equation}
where \(P_{y,w}^*\) is obtained from \(P_{y,w}^J\) by replacing \(q\)
by~\(q^{-2}\).
\end{prop}

The proof that \(\mathscr{I}=D_J\) is a \(W\!\)-graph ideal with
respect to~\(J\) is very similar to the proof just given. We find
that
\begin{align*}
\SA(w)&=\{\,s\in S\mid sw>w\text{ and }sw\in D_J\,\},\\
\SD(w)&=\{\,s\in S\mid sw<w\,\},\\
\WA_J(w)&=\{\,s\in S\mid sw\in D_J\setminus D_J\,\}=\emptyset,\\
\noalign{\vskip-2pt\hbox{and}\vskip-2pt}
\WD_J(w)&=\{\,s\in S\mid sw\notin D_J\,\}\\
& =\{\,s\in S\mid sw=wt\text{ for some }t\in J\,\}.
\end{align*}
Thus the weak ascents of the previous case are now weak descents,
and vice versa. The corresponding \(\mathcal{H}\)-module is
\(\mathscr{S}_{\!\phi}=\mathcal{H}\otimes_{\mathcal{H}_J}\mathcal{A}_\phi\),
where \(\mathcal{A}_\phi\) is  \(\mathcal{A}\) made into an
\(\mathcal{H}_J\)-module via the homomorphism
\(\phi\colon\mathcal{H}_J\to\mathcal{A}\) that satisfies
\(\phi(T_u)=(-q)^{-l(u)}\) for all \(u\in W_J\). This corresponds
to \(M^J\) in \cite{deo:paraKL} in the case \(u=-1\). We again
define \(b_w=T_w\otimes 1\) for all~\(w\in D_J\), and this time we
find that
\[
T_sb_w=\begin{cases}
b_{sw}&\text{if \(w\in\SA(w)\)}\\
b_{sw}+(q-q^{-1})b_w&\text{if \(w\in\SD(w)\)}\\
-q^{-1}b_w&\text{if \(w\in\WD_J(w)\)}
\end{cases}
\]
in accordance with the requirements of
Definition~\ref{wgphdetelt}. The proof that
\(\mathscr{S}_{\!\phi}\) admits an \(\mathcal{A}\)-semilinear
involution with the required properties is exactly as in the
previous case. Again the \(W\!\)-graph basis given by our
construction is essentially the same as the basis of \(M^J\) in
Proposition 3.2 (iii) of \cite{deo:paraKL} (now in the case
\(u=-1\)).

\begin{prop}\label{deo-case2}
The \(\mathcal H\)-module \(\mathscr{S}_{\!\phi}\) has a \(W\!\)-graph
basis \((\,c_w\mid w\in D_J\,)\) such that \(\overline{c_w}=c_w\) and
\(c_w=b_w-\sum_{y<w}qp_{y,w}^Jb_y\) for all \(w\in W\), where
\(p_{y,w}^J\) is a polynomial of degree at
most \(l(w)-l(y)-1\) and the W-graph edge-weight \(\mu_{y,w}\) is
the constant term of~\(p_{y,w}^J\). The polynomials \(p_{y,w}^J\)
are related to Deodhar's polynomials \(P_{y,w}^J\) via
\begin{equation}\label{Deorelation2}
p_{y,w}^J = (-q)^{l(w) - l(y) - 1}P_{y,w}^*.
\end{equation}
where \(P_{y,w}^*\) is obtained from \(P_{y,w}^J\) by replacing \(q\)
by~\(q^{-2}\).
\end{prop}

We remark that the above constructions are special cases of
a more general construction to be described in the next section.
If \(J\subseteq S\) and \(J=J_1\cup J_2\), where no element of \(J_1\)
is conjugate in \(W_J\) to any element of \(J_2\), then \(\mathcal{H}_J\)
has a one-dimensional module on whose basis element \(b_1\) the
generators \(T_s\) of \(\mathcal{H}_J\) act as follows:
\[
T_sb_1=\begin{cases}
-q^{-1}b_1&\text{if \(s\in J_1\),}\\
qb_1&\text{if \(s\in J_2\).}
\end{cases}
\]
Thus the subset of \(W_J\) consisting of the identity element alone
is a \(W_J\)-graph ideal with respect to~\(J_1\), with \(\D(1)=\WD(1)=J_1\)
and \(A(1)=\WA(1)=J_2\). By Theorem~\ref{indthr} below it follows that \(D_J\)
is a \(W\!\)-graph ideal with respect to~\(J_1\). Deodhar's two constructions
correspond to the cases \(J_1=\emptyset\) and \(J_1=J\).

\section{Induced \(W\!\)-graph ideals}\label{induced}

Let \(K \subseteq S\), and let \(\mathcal{H}_K\) be the Hecke
algebra associated with the Coxeter system \((W_{K},K)\),
identified with a subalgebra of~\(\mathcal{H}\) as in
Section~\ref{klparrel} above. Suppose that
\(\mathscr{I}_{0}\subseteq W_K\) is a \(W_K\)-graph ideal with
respect to \(J \subseteq K\), and let
\(\mathscr{S}_{0}=\mathscr{S}(\mathscr{I}_{0},J)\) be the
corresponding \(\mathcal{H}_K\)-module. Thus \(\mathscr{S}_{0}\)
has an \(\mathcal{A}\)-basis \((\,b_{z}^{0} \mid z \in \mathscr
I_{0}\,)\) such that for all \(t\in K\) and
\(z\in\mathscr{I}_{0}\),
\begin{equation}\label{S_0action}
T_{t}b_{z}^{0} =
\begin{cases}
  b_{tz}^{0}  & \text{if \(t \in \SA(K,z)\),}\\
  b_{tz}^{0} + (q - q^{-1})b_{z}^{0} & \text{if \(t \in \SD(K,z)\),}\\
  -q^{-1}b_{z}^{0} & \text{if \(t \in \WD_{J}(K,z)\),}\\
  qb_{z}^{0} - \sum\limits_{\substack{y \in \mathscr I_0\\y < tz}} r^t_{y,z}b_{y}^{0} &
  \text{if \(t \in \WA_{J}(K,z)\),}
\end{cases}
\end{equation}
for some \(r^t_{y,z}\in q\mathcal{A}^+\), where the descent and
ascent sets are given by
\begin{align*}
\SA(K,z) &= \{\,t \in K \mid tz > z \text{ and } tz \in \mathscr I_0\,\},\\
\SD(K,z) &= \{\,t \in K \mid tz < z\,\},\\
\WA_{J}(K,z) &= \{\,t \in K \mid tz \notin \mathscr{I}_0\text{ and }z^{-1}tz\notin J\,\},\\
\WD_{J}(K,z) &= \{\,t \in K \mid tz \notin \mathscr{I}_0\text{ and
}z^{-1}tz\in J\,\}.
\end{align*}
Furthermore, \(\mathscr{S}_0\) admits an
\(\mathcal{A}\)-semilinear involution \(\alpha \mapsto
\overline{\alpha}\) satisfying \(\overline{b_1^0}=b_1^0\) and
\(\overline{h\alpha} =\overline{h}\overline{\alpha}\) for all
\(h\in \mathcal{H}_K\) and \(\alpha \in \mathscr{S}_0\).

We shall show that \(\mathscr I=D_{K}\mathscr I_0 = \{\,dz \mid d
\in D_{K}\text{ and }z\in \mathscr I_0\,\}\) is a \(W\!\)-graph
ideal with respect to~\(J\). The corresponding
\(\mathcal{H}\)-module \(\mathscr{S}(\mathscr{I},J)\) is
\(\mathscr{S}=\mathcal{H}\otimes_{\mathcal{H}_K}\mathscr{S}_0\).

\begin{lem}\label{invOrderIdeal}
The set \(\mathscr I\)  defined above is an ideal of \((W,
\leq_{L})\).
\end{lem}
\begin{proof}
In view of Definition~\ref{leftorder}, it suffices to show that
\(sw\in D_{K}\mathscr I_0\) whenever \(s\in S\) and \(w\in
D_{K}\mathscr I_0\) satisfy \(l(sw)<l(w)\).

Let \(w=dz\), where \(d\in D_{K}\) and \(z\in\mathscr I_0\). Let
\(s\in S\), and suppose that \(l(sw)<l(w)\). If \(sd\in D_K\) then
trivially \(sw=(sd)z\in D_{K}\mathscr I_0\). Now suppose that
\(sd\notin D_K\). By Lemma \ref{deo1} this gives \(sd=dt\) for
some \(t\in K\), and since \(z\in \mathscr{I}_0\subseteq W_K\) we
see that \(tz\in W_K\). Hence, since \(d\in D_K\),
\[
l(tz)=l(dtz)-l(d)=l(sdz)-l(d)=l(sw)-l(d)<l(w)-l(d)=l(dz)-l(d)=l(z).
\]
Since \(t\in K\) and \(z\in\mathscr{I}_0\), and \(\mathscr{I}_0\)
is an ideal of \((W_K,\le_L)\), it follows that
\(tz\in\mathscr{I}_0\). Hence \(sw=d(tz)\in D_{K}\mathscr I_0\) in
this case also, as required.
\end{proof}

For each \(w\in\mathscr I\) the sets of strong ascents, strong
descents, weak ascents and weak descents of~\(w\) relative to
\(\mathscr I\) and \(J\) are defined as in Section~\ref{section5}
above. Note that each \(w\in W\) is uniquely expressible as \(dz\)
with \(d\in D_K\) and \(z\in W_K\), and  \(w\in\mathscr I\) if and
only if \(z\in\mathscr{I}_0\). Moreover,
\[
\mathscr{S}=\bigoplus_{d\in
D_K}T_d\mathcal{H}_K\otimes_{\mathcal{H}_K}\mathscr{S}_0
=\bigoplus_{d\in D_K}T_d\otimes\mathscr{S}_0
\]
and it follows that \(\mathscr{S}\) is \(\mathcal A\)-free with
\(\mathcal A\)-basis \((\,T_d\otimes b_z^0\mid d\in D_K\text{ and
}z\in\mathscr{I}_0\,)\). We define \(b_w=T_d\otimes b_z^0\)
whenever \(w=dz\) as above, and proceed to show that for each
\(s\in S\) and \(w\in\mathscr{I}\) the generator \(T_s\) of
\(\mathcal H\) acts on the basis element \(b_w\) in accordance
with Definition~\ref{wgphdetelt}.

Let \(w=dz\), where \(d\in D_K\) and \(z\in \mathscr{I}_0\), and
let \(s\in\SA(w)\), so that \(w<sw\in\mathscr{I}\). Suppose first
that \(sd\notin D_K\), so that \(d<sd=dt\) for some \(t\in K\), by
Lemma~\ref{deo1}. Then \(tz\in W_K\), and since \(d(tz)=sw\in
D_K\mathscr{I}_0\), it follows that \(tz\) must be in
\(\mathscr{I}_0\). Moreover, since \(l(w)<l(sw)\),
\[
l(tz)=l(d(tz))-l(d)=l(sw)-l(d)>l(w)-l(d)=l(dz)-l(d)=l(z),
\]
and therefore \(t\in\SA(K,z)\). By \eqref{S_0action} above it
follows that
\[
T_sb_w=T_sT_d\otimes b_z^0=T_dT_t\otimes b_z^0=T_d\otimes
T_tb_z^0=T_d\otimes b_{tz}^0=b_{dtz}=b_{sw}
\]
in accordance with Definition~\ref{wgphdetelt}. It remains to show
that this same equation holds if \(sd\in D_K\), and in this case
we find that
\[
b_{sw}=b_{(sd)z}=T_{sd}\otimes b_z^0=T_sT_d\otimes b_z^0=T_sb_w,
\]
as required.

Suppose now that \(s\in\SD(w)\), where \(w=dz\) as above, so that
\(sw<w\). Suppose first that \(sd\notin D_K\), so that \(d<sd=dt\)
for some \(t\in K\), by Lemma~\ref{deo1}. Then \(tz\in W_K\), and
since \(l(w)<l(sw)\) it follows that
\[
l(tz)=l(d(tz))-l(d)=l(sw)-l(d)<l(w)-l(d)=l(dz)-l(d)=l(z),
\]
whence \(t\in\SD(K,z)\). By \eqref{S_0action},
\[
\begin{split}
T_sb_w&=T_sT_d\otimes b_z^0=T_dT_t\otimes b_z^0=T_d\otimes T_tb_z^0=T_d\otimes(b_{tz}^0+(q-q^{-1})b_z^0)\\
&\ =(T_d\otimes b_{tz}^0)+(q-q^{-1})(T_d\otimes
b_z^0)=b_{dtz}+(q-q^{-1})b_{dz}=b_{sw}+(q-q^{-1})b_{w}
\end{split}
\]
in accordance with Definition~\ref{wgphdetelt}. It remains to show
that this same equation holds if \(sd\in D_K\). In this case
\(b_{sw}=b_{(sd)z}=T_{sd}\otimes b_z^0\), and we also find that
\(l(sd)=l((sd)z)-l(z)=l(sw)-l(z)<l(w)-l(z)=l(dz)-l(z)=l(d)\). So
\[
\begin{split}
T_sb_w=T_sT_d&\otimes b_z^0=(T_{sd}+(q-q^{-1})T_d)\otimes b_z^0\\
&=(T_{sd}\otimes b_z^0)+(q-q^{-1})(T_d\otimes
b_z^0)=b_{sw}+(q-q^{-1})b_w
\end{split}
\]
as required.

Next, suppose that \(s\in\WD_J(w)\), where \(w=dz\) as above, so
that \(sw\notin\mathscr{I}\) and \(w^{-1}sw\in J\). Since
\(sw=(sd)z\) and \(z\in\mathscr{I}_0\), the fact that
\(sw\notin\mathscr{I}=D_K\mathscr{I}_0\) means that \(sd\notin
D_K\), and so \(sd=dt\) for some \(t\in K\), by Lemma~\ref{deo1}.
Moreover, \(z^{-1}tz=z^{-1}d^{-1}sdz=w^{-1}sw\in J\), so that
\(t\in\WD_J(K,z)\). By \eqref{S_0action},
\[
T_sb_w=T_sT_d\otimes b_z^0=T_dT_t\otimes b_z^0=T_d\otimes
T_tb_z^0=T_d\otimes (-q^{-1})b_{z}^0 =-q^{-1}b_{w}
\]
in accordance with Definition~\ref{wgphdetelt}.

Finally, suppose that \(s\in\WA_J(w)\), where \(w=dz\) as above,
so that \(sw\notin\mathscr{I}\) and \(w^{-1}sw\notin J\). As in
the preceding case it follows that \(sd\notin D_K\), and \(sd=dt\)
for some \(t\in K\), but now \(z^{-1}tz=w^{-1}sw\notin J\). So
\(t\in\WA_{J}(K,z)\), and by \eqref{S_0action} it follows that
\(T_tb_z^0=qb_{z}^{0} - \sum_y r^t_{y,z}b_{y}^{0}\) for some
polynomials \(r_{y,z}^t\in q\mathcal{A}^+\) (defined whenever
\(y<tz\) and \(y\in\mathscr{I}_0\)). Hence
\[
\begin{split}
T_sb_w=T_sT_d&\otimes b_z^0=T_dT_t\otimes b_z^0=T_d\otimes
T_tb_z^0
=T_d\otimes\bigl(qb_{z}^{0}-\sum_y r^t_{y,z}b_{y}^{0}\bigr)\\
&=q(T_d\otimes b_z^0)-\sum_y r^t_{y,z}(T_d\otimes b_{y}^{0})
=qb_{w}-\sum\limits_y r^t_{y,z}b_{dy}
\end{split}
\]
where the sums range over \(y\in\mathscr{I}_0\) such that
\(y<tz\). Since \(y\in\mathscr{I}_0\) and \(y<tz\) imply that
\(dy\in D_K\mathscr{I}_0=\mathscr{I}\) and \(dy<dtz=sw\) (by
Lemma~\ref{lifting1} and an induction on~\(l(d)\)), we conclude
that in this case also the requirements of
Definition~\ref{wgphdetelt} are satisfied.

To complete the proof that \(\mathscr I\) is a \(W\!\)-graph ideal
with respect to \(J\) it remains only to show that \(\mathscr S\)
admits a semilinear involution \(\alpha\mapsto\overline\alpha\)
such that \(\overline{h\alpha}=\overline{h}\overline{\alpha}\) for
all \(h\in\mathcal H\) and \(\alpha\in\mathscr S\!\). The proof is
very similar to the corresponding proofs in Section~\ref{klparrel}
above: we set \(\overline{T_d\otimes
b_z^0}=\overline{T_d}\otimes\overline{b_z^0}\) for all \(d\in
D_K\) and \(z\in\mathscr{I}_0\), using semilinearity to extend the
definition to the whole of~\(\mathscr{S}\). We omit further
details.

The discussion above enables us to state the following theorem.

\begin{thr}\label{indthr}
Let \(K \subseteq S\) and suppose that \(\mathscr{I}_{0}\subseteq
W_K\) is a \(W_K\)-graph ideal with respect to \(J \subseteq K\),
and let \(\mathscr{S}_{0}=\mathscr{S}(\mathscr{I}_{0},J)\) be the
corresponding
 \(\mathcal{H}_K\)-module. Then \(\mathscr{I}=D_K\mathscr{I}_0\)
is a \(W\!\)-graph ideal with respect to~\(J\), the corresponding
\(\mathcal{H}\)-module \(\mathscr{S}(\mathscr{I},J)\) being
isomorphic to \(\mathcal{H}\otimes_{\mathcal H_K}\mathscr{S}_0\).
\end{thr}

\begin{rem}
In the situation of Theorem~\ref{indthr}, the assumption that
\(\mathscr{I}_0\) is a \(W_K\)-graph ideal in \((W_K,\le_L)\)
implies, by the construction in Section~\ref{section7}, that
\(\mathscr{S}_0\) is isomorphic to an \(\mathcal H_K\)-module
arising from a \(W_K\)-graph. By \cite[Theorem
5.1]{howyin:indwgraph} it follows that the induced
module~\(\mathscr S\) is isomorphic to a \(W\!\)-graph module.
Theorem~\ref{indthr} yields an alternative construction of the
induced~\(W\!\)-graph in this special case that the \(W_K\)-graph
in question comes from a \(W_K\)-graph ideal in \((W_K,\le_L)\).
\end{rem}

\end{document}